\documentclass{amsart} 

\usepackage{lmodern} 
\usepackage{microtype} 
\usepackage[english]{babel} 
\usepackage{hyperref} 
\usepackage{mathrsfs} 
\usepackage[msc-links]{amsrefs} 
\usepackage{xcolor} 
\usepackage{pgfplots} 
\usepgfplotslibrary{fillbetween} 
\pgfplotsset{compat=1.18}

\theoremstyle{plain} 
\newtheorem{theorem}{Theorem}[section] 
\newtheorem{lemma}[theorem]{Lemma} 
\newtheorem{corollary}[theorem]{Corollary}

\numberwithin{equation}{section} 

\DeclareMathOperator{\mre}{Re} 
\DeclareMathOperator{\mim}{Im} 
\DeclareMathOperator{\dist}{dist} 
\DeclareMathOperator{\supp}{supp} 
\DeclareMathOperator*{\esssup}{\operatorname{ess\,sup}}

\begin{document} 
\title{Almost periodicity and boundary values of Dirichlet series} 
\date{\today}

\author{Ole Fredrik Brevig} 
\address{Department of Mathematics, University of Oslo, 0851 Oslo, Norway} 
\email{obrevig@math.uio.no}

\author{Athanasios Kouroupis}
\address{Department of mathematics, KU Leuven, Celestijnenlaan 200B, 3001, Leuven, Belgium}
\email{athanasios.kouroupis@kuleuven.be}

\author{Karl-Mikael Perfekt} 
\address{Department of Mathematical Sciences, Norwegian University of Science and Technology (NTNU), 7491 Trondheim, Norway} 
\email{karl-mikael.perfekt@ntnu.no}
\begin{abstract}
	We employ almost periodicity to establish analogues of the Hardy--Stein identity and the Littlewood--Paley formula for Hardy spaces of Dirichlet series. A construction of Saksman and Seip shows that the limits in this Littlewood--Paley formula cannot be interchanged. We apply this construction to show that the limits in the definition of the mean counting function for Dirichlet series cannot be interchanged. These are essentially statements about the two different kinds of boundary values that we associate with Dirichlet series that converge to a bounded analytic function in a half-plane. The treatment of the mean counting function also involves an investigation of the zero sets and Blaschke products of such Dirichlet series. 
\end{abstract}


\maketitle

\section{Introduction} Let $\mathscr{H}^\infty$ denote the collection of all bounded analytic functions in the right half-plane that can be represented as a convergent Dirichlet series 
\begin{equation}\label{eq:diriseri} 
	f(s) = \sum_{n=1}^\infty a_n n^{-s} 
\end{equation}
in the half-plane $\mathbb{C}_\kappa = \{s=\sigma+it\,:\, \sigma>\kappa\}$ for \emph{some} $\kappa>0$ and equip $\mathscr{H}^\infty$ with the supremum norm. A celebrated theorem of Harald Bohr~\cites{Bohr1913A,BrK2024} asserts that if $f$ is in $\mathscr{H}^\infty$, then the Dirichlet series \eqref{eq:diriseri} converges uniformly to $f$ in $\mathbb{C}_\kappa$ for \emph{every} $\kappa>0$. In particular, this means that $f$ is almost periodic in $\mathbb{C}_\kappa$ for every $\kappa>0$.

A basic application of the almost periodicity is that if $f$ is in $\mathscr{H}^\infty$, then the mean values 
\begin{equation}\label{eq:Mpmean} 
	M_p^p(f,\sigma_0) = \lim_{T\to\infty} \frac{1}{2T}\int_{-T}^T |f(\sigma_0+it)|^p \,dt 
\end{equation}
exist for every $1 \leq p<\infty$ and every $\sigma_0>0$. In analogy with the corresponding quantities for Hardy spaces in the unit disc, we set 
\begin{equation}\label{eq:Hpnorm} 
	\|f\|_{\mathscr{H}^p} = \lim_{\sigma_0 \to 0^+} M_p(f,\sigma_0). 
\end{equation}
The first main results of the present paper are analogues of the Hardy--Stein identity and the Littlewood--Paley formula in our setting. 
\begin{theorem}[Hardy--Stein identity] \label{thm:HS} 
	Fix $1 \leq p<\infty$. If $f$ is in $\mathscr{H}^\infty$, then the function $\kappa \mapsto M_p^p(f,\kappa)$ is continuously differentiable on $(0,\infty)$ and
	\[\frac{\partial}{\partial \kappa} M_p^p(f,\kappa) = -\lim_{T\to\infty} \frac{p^2}{2T} \int_\kappa^\infty \int_{-T}^T |f(s)|^{p-2} |f'(s)|^2 \,dtd\sigma.\]
	The limit converges uniformly on $(\kappa_0,\infty)$ for every fixed $\kappa_0>0$. 
\end{theorem}

It follows from \cite{BP2021}*{Theorem~3.8} that $\kappa \mapsto M_p(f,\kappa)$ is a logarithmically convex and decreasing function: the latter statement is made precise by Theorem~\ref{thm:HS}.

The strategy for our proof of Theorem~\ref{thm:HS} is essentially the same as the strategy used by P. Stein \cite{Stein1933} for the proof of the classical Hardy--Stein formula in the unit disc. However, since half-planes are unbounded we will require information about distribution of the zero sets of functions in $\mathscr{H}^\infty$. This information will be extracted from the almost periodicity by way of results due to Bohr and Jessen \cite{BJ1930}.

Since $\kappa \mapsto M_p^p(f,\kappa)$ is continuously differentiable, we can write 
\begin{equation}\label{eq:diffint} 
	M_p^p(f,\sigma_1) - M_p^p(f,\sigma_0) = \int_{\sigma_0}^{\sigma_1} \frac{\partial}{\partial \kappa} M_p^p(f,\kappa)\,d\kappa 
\end{equation}
for $\sigma_1>\sigma_0>0$. It is plain that if $f$ is the Dirichlet series \eqref{eq:diriseri}, then $f(s)$ converges uniformly to $a_1$ as $\mre{s} \to \infty$. Setting $f(+\infty)=a_1$, we have $M_p^p(f,\sigma_1) \to |f(+\infty)|^p$ as $\sigma_1 \to \infty$. The following result will be obtained from Theorem~\ref{thm:HS} via \eqref{eq:diffint}.
\begin{corollary}[Littlewood--Paley formula] \label{cor:LP} 
	Fix $1 \leq p<\infty$. If $f$ is in $\mathscr{H}^\infty$, then
	\[\|f\|_{\mathscr{H}^p}^p = |f(+\infty)|^p + \lim_{\sigma_0 \to 0^+} \lim_{T\to\infty} \frac{p^2}{2T} \int_{\sigma_0}^\infty \int_{-T}^T |f(s)|^{p-2} |f'(s)|^2 \,(\sigma-\sigma_0)\,dtd\sigma.\]
\end{corollary}

Note that the case $p=2$ of the Littlewood--Paley formula was first established by Bayart \cite{Bayart2003}*{Proposition~2} and that a version of Corollary~\ref{cor:LP} without the precise constant was obtained by Bayart, Queff\'{e}lec, and Seip \cite{BQS2016}*{Theorem~5.1}. 

Since the elements of $\mathscr{H}^\infty$ are bounded analytic functions in the right half-plane, the boundary values 
\begin{equation}\label{eq:bv} 
	f(i\tau) = \lim_{\sigma\to 0^+} f(\sigma+i\tau) 
\end{equation}
exist for almost every real number $\tau$. It is natural to ask how much of the almost periodicity of $f$ is carried to the boundary values. This question has several interpretations, and ours is primarily inspired by work of Saksman and Seip \cite{SS2009}.
\begin{theorem}[Saksman--Seip] \label{thm:SS} 
	Fix $1 \leq p<\infty$. 
	\begin{enumerate}
		\item[(a)] For every $0<\varepsilon<1$, there is a function $f$ in $\mathscr{H}^\infty$ such that $\|f\|_{\mathscr{H}^p} = \varepsilon$ and such that
		\[\lim_{T\to\infty} \frac{1}{2T} \int_{-T}^T |f(i\tau)|^p \,d\tau = 1.\]
		\item[(b)] There is a function $f$ in $\mathscr{H}^\infty$ such that
		\[\lim_{T\to \infty} \frac{1}{2T} \int_{-T}^T |f(i \tau)|^p \,dt\]
		does not exist. 
	\end{enumerate}
\end{theorem}

From a historical perspective, let us mention that Theorem~\ref{thm:SS}~(a) in a slightly different setting was obtain by Besicovitch~\cite{Besicovitch1927} by completely different techniques. 

The underlying issue in Theorem~\ref{thm:SS} is that the set of boundary values \eqref{eq:bv} is too small to capture in a complete way the almost periodicity of $f$. The correct approach is through the boundary values of the vertical limit functions of $f$. From this point of view, the almost periodicity of $f$ in $\mathbb{C}_\kappa$ for every $\kappa>0$ is carried over to the ergodicity of the Kronecker flow on the imaginary axis. This point will be elucidated in the preliminary Section~\ref{sec:bvbv} below.

We may interpret Theorem~\ref{thm:SS} as saying that the limits in $\sigma_0$ and $T$ in \eqref{eq:Mpmean} and \eqref{eq:Hpnorm} may in general not be interchanged without damage to the result. It is natural to wonder whether the same phenomenon holds for the Littlewood--Paley formula in Corollary~\ref{cor:LP}. This is indeed the case, as follows directly from Theorem~\ref{thm:SS} and the following result.
\begin{theorem}\label{thm:LPswap} 
	Fix $1 \leq p<\infty$ and suppose that $f$ is in $\mathscr{H}^\infty$. Then
	\[\lim_{T\to\infty} \left|\frac{1}{2T} \int_{-T}^T |f(i\tau)|^p \,dt - |f(+\infty)|^p - \frac{p^2}{2T}\int_0^\infty \int_{-T}^T |f(s)|^{p-2} |f'(s)|^2 \,\sigma dt d\sigma \right|=0.\]
\end{theorem}

This illustrates in a perhaps stronger sense how the almost periodicity of functions in $\mathscr{H}^\infty$ fails to extend from $\mathbb{C}_\kappa$ for every $\kappa>0$ to $\mathbb{C}_0$: we are not in Corollary~\ref{cor:LP} concerned with the boundary values, but rather the order of integration in $\mathbb{C}_0$.

If $f$ is in $\mathscr{H}^\infty$, then the \emph{mean counting function}\footnote{Our definition of the mean counting function in \eqref{eq:meancounting} differs from the definition presented in \cite{BP2021}, where $(\mre{s}-\sigma_0)$ is replaced by $\mre{s}$. The present definition is more natural from the point of view of both the Littlewood--Paley formula and Littlewood's argument principle, and is therefore easier to analyze. The two competing definitions are equivalent due to \cite{BP2021}*{Theorem~6.4}, but this requires a deep result of Jessen and Tornehave \cite{JT1945}. } 
\begin{equation}\label{eq:meancounting} 
	\mathscr{M}_f(\xi) = \lim_{\sigma_0 \to 0^+} \lim_{T\to\infty} \frac{\pi}{T} \sum_{\substack{s \in f^{-1}(\{\xi\}) \\
	|\mim{s}| < T \\
	\sigma_0<\mre{s}<\infty}} \left(\mre{s}-\sigma_0\right), 
\end{equation}
is well-defined for every $\xi$ in $\mathbb{D} \setminus \{f(+\infty)\}$, as demonstrated recently in \cite{BP2021}. It plays the same fundamental role as the Nevanlinna counting function does in the classical theory, in connection for example with the Nevanlinna class, formulas of Jensen-type, and composition operators. See the main results of \cite{BP2021} and \cite{Shapiro1987}, as well as Theorem~\ref{thm:Nfswap} below. If $f$ maps $\mathbb{C}_0$ to the unit disc $\mathbb{D}$, then the mean counting function enjoys the same pointwise Littlewood-type estimate 
\begin{equation}\label{eq:Lest} 
	\mathscr{M}_f(\xi) \leq \log\left|\frac{1-\overline{\xi}f(+\infty)}{\xi-f(+\infty)}\right| 
\end{equation}
as the Nevanlinna counting function.

It was enquired in \cite{BP2021}*{Problem~1} whether the formula \eqref{eq:meancounting} is true if the limits in $\sigma_0$ and $T$ are interchanged. Suppose that $f$ is a Dirichlet series mapping $\mathbb{C}_0$ to $\mathbb{D}$. Setting 
\[N_f(\xi,T) = \frac{\pi}{T} \sum_{\substack{s \in f^{-1}(\{\xi\}) \\
|\mim{s}| < T \\
0<\mre{s}<\infty}}\mre{s},\]
it is not difficult to see that this problem can be reformulated as to ask whether the limit 
\begin{equation}\label{eq:limit} 
	\lim_{T\to\infty} N_f(\xi,T) 
\end{equation}
equals $\mathscr{M}_f(\xi)$. It follows directly from the definitions of $\mathscr{M}_f$ and $N_f$ that 
\begin{align*}
	\mathscr{M}_f(\xi) \leq &\liminf_{T\to \infty} N_f(\xi,T), \intertext{and we shall establish in Theorem~\ref{thm:limsup} below that} &\limsup_{T\to \infty} N_f(\xi,T) \leq \log\left|\frac{1-\overline{\xi}f(+\infty)}{\xi-f(+\infty)}\right| 
\end{align*}
for $\xi$ in $\mathbb{D} \setminus \{f(+\infty)\}$. Consequently, if the Littlewood-type estimate \eqref{eq:Lest} is attained for some $\xi$, then the limit \eqref{eq:limit} equals $\mathscr{M}_f(\xi)$. Recall from \cite{BP2021}*{Theorem~6.6} that if the Littlewood-type estimate \eqref{eq:Lest} is attained for \emph{one} $\xi$, then it is attained for \emph{quasi-every} $\xi$, and $f$ must be inner with respect to $\mathscr{H}^\infty$, in the sense that $\|f\|_{\mathscr{H}^p} = \|f\|_{\mathscr{H}^\infty}$ for all $1 \leq p<\infty$.

To state the analogue of Theorem~\ref{thm:LPswap} for the counting function $N_f(\xi,T)$, we recall that if $f$ is an analytic function that maps the right half-plane to the unit disc, then the \emph{Frostman shifts} of $f$ are 
\begin{equation}\label{eq:frostman} 
	f_\xi(s) = \frac{\xi-f(s)}{1-\overline{\xi}f(s)} 
\end{equation}
for $\xi$ in $\mathbb{D}$. The key point is Frostman's theorem or---more precisely---Rudin's generalization of Frostman's theorem \cite{Rudin1967}, which asserts that $f_\xi$ lacks a singular inner part for quasi-every $\xi$ in $\mathbb{D}$.
\begin{theorem}\label{thm:Nfswap} 
	If $f$ is a Dirichlet series mapping $\mathbb{C}_0$ to $\mathbb{D}$, then
	\[\lim_{T\to\infty} \left|N_f(\xi,T) - \frac{1}{2T}\int_{-T}^T \log|f_\xi(i\tau)| \,d\tau -\log\left|\frac{1-\overline{\xi}f(+\infty)}{\xi-f(+\infty)}\right|\right|=0\]
	for quasi-every $\xi$ in $\mathbb{D}\setminus\{f(+\infty)\}$. 
\end{theorem}

The proof of Theorem~\ref{thm:Nfswap} relies on a new result (Theorem~\ref{thm:Hinftyloc} below) asserting that the zero sets of functions in $\mathscr{H}^\infty$ enjoy a strong version of the classical Blaschke condition. In contrast to the results about the zero sets of functions in $\mathscr{H}^\infty$ mentioned above, this result does not stem from the almost periodicity of functions in $\mathscr{H}^\infty$, but rather from their behavior as $\mre{s} \to \infty$. We will compile the various results about the zero sets of functions in $\mathscr{H}^\infty$ in Section~\ref{sec:zeroset}.

Armed with Theorem~\ref{thm:Nfswap} and the construction of Saksman and Seip (presented in Theorem~\ref{thm:SSconstruct} below), we will establish the two following results.
\begin{theorem}\label{thm:largelim} 
	Let $0<\varepsilon<1$ be fixed. There is a Dirichlet series $f$ mapping $\mathbb{C}_0$ to $\mathbb{D}$ with $f(+\infty)=0$ such that 
	\begin{equation}\label{eq:largelim2} 
		\lim_{T\to\infty} N_f(\xi,T) = \log{\frac{1}{|\xi|}} 
	\end{equation}
	for quasi-every $\xi$ in $\mathbb{D} \setminus \{0\}$, yet 
	\begin{equation}\label{eq:largelim1} 
		\mathscr{M}_f(\xi) \leq \log{\frac{1}{|\xi|}} - (1-\varepsilon) \frac{1-|\xi|^2}{2}, 
	\end{equation}
	for every $\xi$ in $\mathbb{D}\setminus\{0\}$. 
\end{theorem}
In comparing \eqref{eq:largelim2} and \eqref{eq:largelim1}, it may be useful to note that that \eqref{eq:largelim1} in particular yields that
\[\limsup_{|\xi| \to 1^{-}} \frac{\mathscr{M}_f(\xi)}{-\log{|\xi|}} \leq \varepsilon.\]
\begin{theorem}\label{thm:nolim} 
	For any $0<\varepsilon<1$, there is a Dirichlet series $f$ mapping $\mathbb{C}_0$ to $\mathbb{D}$ such that the limit
	\[\lim_{T\to\infty} N_f(\xi,T)\]
	fails to exist for quasi-every $\xi$ such that $\varepsilon < |\xi| < 1$. 
\end{theorem}

There is nothing special about the assertion that $f(+\infty)=0$ in Theorem~\ref{thm:largelim} or the role of $\xi = 0$ in Theorem~\ref{thm:nolim}: we have made these choices to simplify the expressions as much as possible. In any case, it is not possible to improve \eqref{eq:largelim1} close to $\xi = f(+\infty)$, since
\[\mathscr{M}_f(\xi) \sim -\log|\xi-f(+\infty)|\]
as $\xi \to f(+\infty)$ for any non-constant Dirichlet series $f$ mapping $\mathbb{C}_0$ to $\mathbb{D}$. However, it would be interesting to see a version of Theorem~\ref{thm:nolim} with $\varepsilon=0$, necessitating a more elaborate construction than ours.

Theorems~\ref{thm:SS}, \ref{thm:LPswap}, \ref{thm:largelim}, and \ref{thm:nolim} demonstrate that various formulas in $\mathscr{H}^p$ theory hold only when the $T$-limit is computed before the $\sigma_0$-limit. As discussed above, this is because the boundary values \eqref{eq:bv} do not retain the almost periodicity of the functions in question. However, if we consider the larger set of generalized boundary values (see Section~\ref{sec:bvbv}), then the corresponding limits may be interchanged. See Theorem~\ref{thm:Hpchiswap}, Theorem~\ref{thm:LPchiswap}, and Theorem~\ref{thm:mcfchiswap} below for the precise statements.

Theorem~\ref{thm:largelim} and Theorem~\ref{thm:nolim} resolve \cite{BP2021}*{Problem~1}. We close this introduction by mentioning that the recent papers \cite{Bayart2024} and \cite{BK2024} have made progress on, respectively, the two other related problems \cite{BP2021}*{Problem~3} and \cite{BP2021}*{Problem~2}.

\subsection*{Organization} The present paper is divided into five further sections, the two first of which are of a somewhat preliminary nature. In Section~\ref{sec:bvbv} we collate some known results about the two types of boundary values a function in $\mathscr{H}^\infty$ has and set out their interactions through the ergodic theorem. As mentioned above, Section~\ref{sec:zeroset} contains some old and new results on the zero sets of functions in $\mathscr{H}^\infty$. The Hardy--Stein identity and the Littlewood--Paley formula is the main topic of Section~\ref{sec:LP}. It is here the proofs of Theorem~\ref{thm:HS}, Corollary~\ref{cor:LP}, and Theorem~\ref{thm:LPswap} may be found. Section~\ref{sec:jensen} contains some expository material on the mean counting function from \cite{BP2021} and culminates with the proof of Theorem~\ref{thm:Nfswap}. The final Section~\ref{sec:proofproof} is devoted to the proofs of Theorem~\ref{thm:largelim} and Theorem~\ref{thm:nolim}.

\section{Boundary values and boundary values} \label{sec:bvbv} 
For real numbers $\tau$, consider the \emph{vertical translation}
\[V_\tau f(s) = f(s+i\tau).\]
If $f$ is in $\mathscr{H}^\infty$ and if $(\tau_k)_{k\geq1}$ is a sequence of real numbers, then improved Montel theorem for $\mathscr{H}^\infty$ due to Bayart \cite{Bayart2002}*{Lemma~18} asserts that we may pass to a subsequence of $(\tau_k)_{k\geq1}$ to ensure that $V_{\tau_k}f$ converges uniformly in $\mathbb{C}_\kappa$ for every $\kappa>0$ to a function $g$ in $\mathscr{H}^\infty$ with $\|g\|_{\mathscr{H}^\infty}=\|f\|_{\mathscr{H}^\infty}$. We will call a Dirichlet series $g$ obtained in this way a \emph{vertical limit function} of $f$.

A \emph{character} $\chi$ is a completely multiplicative map from $\mathbb{N}$ to the unit circle $\mathbb{T}$. The set of characters can be identified with the infinite-dimensional torus
\[\mathbb{T}^\infty = \mathbb{T} \times \mathbb{T} \times \mathbb{T} \times \cdots,\]
since $\chi$ is determined by its value at the prime numbers. As $\mathbb{T}^\infty$ is a compact abelian group under coordinate-wise multiplication, it comes with a unique normalized Haar measure that we will denote $m_\infty$.

The following description of the vertical limit functions is a consequence of Kronecker's theorem (see e.g. \cite{HLS1997}*{Section~2.3} or \cite{QQ2020}*{Section~2.2.3}).
\begin{lemma}\label{lem:vlf} 
	The vertical limit functions of $f$ of the form \eqref{eq:diriseri} in $\mathscr{H}^\infty$ coincide with the functions
	\[f_\chi(s) = \sum_{n=1}^\infty a_n \chi(n) n^{-s},\]
	where $\chi$ is a character. Moreover, $\|f_\chi\|_{\mathscr{H}^\infty} = \|f\|_{\mathscr{H}^\infty}$ for every character $\chi$. 
\end{lemma}

Since $f_\chi$ is in $\mathscr{H}^\infty$ for every $\chi$ on $\mathbb{T}^\infty$, we get from \eqref{eq:bv} that the limit of $f_\chi(\sigma+i\tau)$ as $\sigma \to 0^+$ exists for almost every $\tau$. We find it natural to reformulate this statement by a standard argument involving Fubini's theorem (see e.g.~\cite{SS2009}*{Theorem~2}). 
\begin{lemma}\label{lem:chibv} 
	If $f$ is in $\mathscr{H}^\infty$, then the limit
	\[f^\ast(\chi) = \lim_{\sigma\to 0^+} f_\chi(\sigma)\]
	exists for almost every $\chi$ in $\mathbb{T}^\infty$. 
\end{lemma}

We now have two different sets of boundary values for a function $f$ in $\mathscr{H}^\infty$, namely
\begin{itemize}
	\item the ``natural'' boundary values $f(i\tau)$ for almost every $\tau$ in $\mathbb{R}$; 
	\item the ``generalized'' boundary values $f^\ast(\chi)$ for almost every $\chi$ on $\mathbb{T}^\infty$. 
\end{itemize}

The natural boundary values are a subset of the generalized boundary values, because $f_\chi = V_\tau f$ holds whenever $\chi(p_j)=p_j^{-i\tau}$ for $j=1,2,3,\ldots$. We will write $\chi = \mathfrak{p}^{-i\tau}$ for these characters. Since the generalized boundary values constitute a larger set, it is natural that they are better equipped to fully comprehend the various eccentricities of the function $f$. 

This is, of course, not always the case. The most basic and well-known example is that both sets of boundary values can be used to compute the $\mathscr{H}^\infty$ norm of $f$, i.e. that
\[\sup_{s \in \mathbb{C}_0} |f(s)| = \esssup_{\tau \in \mathbb{R}} |f(i\tau)| = \esssup_{\chi \in \mathbb{T}^\infty} |f^\ast(\chi)|.\]
In both cases, the estimate $\geq$ is trivial while the estimate $\leq$ is obtained using Poisson kernels either in the half-plane or in the polydisc (see e.g. \cite{HLS1997}*{Lemma~2.3}).

Another interesting phenomenon is exemplified by the function
\[f(s) = \exp\left(-\frac{2-2^{-s}-3^{-s}}{2+2^{-s}+3^{-s}}\right).\]
It is not difficult to check that $f$ is continuous in the closed right half-plane $\overline{\mathbb{C}_0}$ and that $f$ belongs to $\mathscr{H}^\infty$. However, the function
\[f^\ast(\chi) = \exp\left(-\frac{2-\chi_1-\chi_2}{2+\chi_1+\chi_2}\right)\]
is not continuous in the point $(\chi_1,\chi_2)=(-1,-1)$. The issue is that this character can only be approached by $\mathfrak{p}^{-i \tau}$ as $\tau \to \infty$ and hence continuity of $f$ does not imply continuity of $f^\ast$. This example is adapted from \cite{CGR2024}*{Proposition~2.9} and we refer to \cite{ABGMN2017}*{Section~2} for further results in this direction.

The interaction between the natural and generalized boundary values can be further explored through the \emph{Kronecker flow} from $\mathbb{R}$ to $\mathbb{T}^\infty$ defined by
\[\tau \mapsto \mathfrak{p}^{-i\tau}.\]
Since the Kronecker flow is ergodic with respect to the Haar measure $m_\infty$ (see e.g.~\cite{CFS1982}*{Section~3.1}), we can apply the ergodic theorem for flows as follows. If $F$ is a continuous function on $\mathbb{T}^\infty$, then
\[\int_{\mathbb{T}^\infty} F(\chi)\,dm_\infty(\chi) = \lim_{T\to\infty} \frac{1}{2T} \int_{-T}^T F(\mathfrak{p}^{-i\tau})\,d\tau.\]
This leads to the following result which, although not explicitly stated in the literature, is well-known.
\begin{lemma}\label{lem:limswapchi1} 
	If $\psi\colon\mathbb{C}\to\mathbb{C}$ is a continuous function and if $f$ is in $\mathscr{H}^\infty$, then
	\[\lim_{\sigma \to 0^+} \lim_{T\to\infty} \frac{1}{2T} \int_{-T}^T (\psi \circ f)(\sigma+it)\,dt = \int_{\mathbb{T}^\infty} (\psi\circ f^\ast)(\chi)\,dm_\infty(\chi).\]
\end{lemma}
\begin{proof}
	For fixed $\sigma>0$, we let $f_\sigma^\ast$ stand for the function on $\mathbb{T}^\infty$ defined by $\chi \mapsto f_\chi(\sigma)$. Since $f$ converges uniformly in $\mathbb{C}_\kappa$ for every $\kappa>0$ and since $f_\chi$ is a vertical limit function of $f$, it is plain that the function $f_\sigma^\ast$ is continuous on $\mathbb{T}^\infty$. It therefore follows from the ergodic theorem that
	\[\lim_{T\to\infty} \frac{1}{2T} \int_{-T}^T (\psi \circ f)(\sigma+it)\,dt = \int_{\mathbb{T}^\infty} (\psi \circ f_\sigma^\ast)(\chi)\,dm_\infty(\chi).\]
	The stated formula follows from this and the dominated convergence theorem. 
\end{proof}

If $F$ is merely integrable on $\mathbb{T}^\infty$, then the ergodic theorem asserts that the formula
\[\int_{\mathbb{T}^\infty} F(\chi)\,dm_\infty(\chi) = \lim_{T\to\infty} \frac{1}{2T} \int_{-T}^T F(\chi' \mathfrak{p}^{-i\tau})\,d\tau\]
holds for almost every $\chi'$ on $\mathbb{T}^\infty$. The same reasoning as above yields the following.
\begin{lemma}\label{lem:limswapchi2} 
	If $\psi\colon\mathbb{C}\to\mathbb{C}$ is a continuous function and if $f$ is in $\mathscr{H}^\infty$, then
	\[\lim_{\sigma \to 0^+} \lim_{T\to\infty} \frac{1}{2T} \int_{-T}^T (\psi \circ f)(\sigma+it)\,dt = \lim_{T\to\infty} \frac{1}{2T} \int_{-T}^T (\psi\circ f_\chi)(i\tau)\,d\tau\]
	for almost every $\chi$ on $\mathbb{T}^\infty$. 
\end{lemma}

It should be noted that the left-hand side of this formula does not change if $f$ is replaced by $f_\chi$ for any $\chi$ on $\mathbb{T}^\infty$. This can be established either by using the fact that $f_\chi$ is a vertical limit function of $f$ or via Lemma~\ref{lem:limswapchi1}. If we consider the specific continuous function $\psi(z)=|z|^p$, then Lemma~\ref{lem:limswapchi1} and Lemma~\ref{lem:limswapchi2} reduce to the following well-known result that should be compared with Theorem~\ref{thm:SS}.
\begin{theorem}\label{thm:Hpchiswap} 
	Fix $1 \leq p<\infty$. If $f$ is in $\mathscr{H}^\infty$, then 
	\begin{align*}
		\|f\|_{\mathscr{H}^p}^p &= \int_{\mathbb{T}^\infty} |f^\ast(\chi)|^p \,dm_\infty(\chi) \intertext{and} \|f\|_{\mathscr{H}^p}^p &= \lim_{T\to \infty} \frac{1}{2T}\int_{-T}^T |f_\chi(i\tau)|^p \,d\tau 
	\end{align*}
	for almost every $\chi$ on $\mathbb{T}^\infty$. 
\end{theorem}

\section{On the zero sets of functions in \texorpdfstring{$\mathscr{H}^\infty$}{Hinfty}} \label{sec:zeroset} 
We will write $\mathscr{Z}_f$ for the zero sequence of a nontrivial analytic function $f$, where we repeat the zeros according to multiplicity. The first results will be consequences of the almost periodicity enjoyed by uniformly convergent Dirichlet series, and we will for completeness state these results in some generality. The standard reference for almost periodic functions is Besicovitch~\cite{Besicovitch1955}.

We will state these results with respect to strips
\[\mathbb{S}_{\alpha,\beta} = \{\sigma+it\,:\, \alpha < \sigma < \beta\}.\]
Recall that a complex-valued function $f$ is said to be \emph{almost periodic} in $\mathbb{S}_{\alpha,\beta}$ if there for every $\varepsilon>0$ is a relatively dense set of real numbers $\tau$ such that 
\begin{equation}\label{eq:ap} 
	|V_\tau f(s)-f(s)| \leq \varepsilon 
\end{equation}
for every $s$ in $\mathbb{S}_{\alpha,\beta}$. It follows from results in \cite{Besicovitch1955}*{\S II.2} that uniformly convergent Dirichlet series are almost periodic in any strip where they converge uniformly. In particular, elements of $\mathscr{H}^\infty$ are uniformly continuous in $\mathbb{S}_{\alpha,\beta}$ for every $0<\alpha<\beta$. (In fact, they satisfy \eqref{eq:ap} in $\mathbb{C}_\kappa$ for every $\kappa>0$.)

It is not difficult to show (see e.g. \cite{OS2008}*{Section~4}) that if $f$ is analytic and almost periodic in $\mathbb{S}_{\alpha,\beta}$, then $f$ either does not vanish in $\mathbb{S}_{\alpha,\beta}$ or there is a number $T$ such that $f$ vanishes in the rectangle
\[(\alpha,\beta) \times [\tau,\tau+T]\]
for every $\tau$. In particular, the function either has no zeros in $\mathbb{S}_{\alpha,\beta}$ or it has an infinite number of zeros in $\mathbb{S}_{\alpha,\beta}$. Note that this does not hold in the boundary of the domain of almost periodicity: the reciprocal of the Riemann zeta function has only one zero in $\overline{\mathbb{C}_1}$. We require some control over the distribution of these zeros and the behavior of $f$ near them.

We begin by recalling three results of Bohr and Jessen~\cite{BJ1930}. A straightforward account of these result can be found in \cite{Levin1980}*{Chapter~VI}.
\begin{lemma}\label{lem:BJ1} 
	Let $f$ be a nontrivial, analytic, and almost periodic function in $\mathbb{S}_{\alpha,\beta}$. For every $\delta>0$, there is a number $m=m(f,\delta)>0$ such that if $s$ lies in the strip $\mathbb{S}_{\alpha+\delta,\beta-\delta}$ and $\dist(s,\mathscr{Z}_f) \geq \delta$, then
	\[|f(s)| \geq m.\]
\end{lemma}
\begin{lemma}\label{lem:BJ2} 
	Let $f$ be a nontrivial, analytic, and almost periodic function in $\mathbb{S}_{\alpha,\beta}$. For every $\delta>0$, there is a number $N=N(f,\delta)$ such that $f$ has at most $N$ zeros, counting multiplicities, in any rectangle
	\[[\alpha+\delta,\beta-\delta]\times[\tau,\tau+1].\]
\end{lemma}
\begin{lemma}\label{lem:BJ3} 
	Let $f$ be a nontrivial, analytic, and almost periodic function in $\mathbb{S}_{\alpha,\beta}$. For $\delta>0$ and $\tau$ in $\mathbb{R}$, set
	\[R_{\delta}(\tau) = [\alpha+\delta/2,\beta-\delta/2]\times[\tau-1/2,\tau+3/2].\]
	For every $\delta>0$, there is $c=c(f,\delta)>0$ such that if $s$ is in $[\alpha+\delta,\beta-\delta]\times[\tau,\tau+1]$, then
	\[|f(s)| \geq c \prod_{\xi \in \mathscr{Z}_f \cap R_\delta(\tau)} |s-\xi|,\]
	where the zeros are repeated according to their multiplicity. 
\end{lemma}

We will also use the following adaptation of these results.
\begin{lemma}\label{lem:dist} 
	Let $f$ be a nontrivial, analytic, and almost periodic function in $\mathbb{S}_{\alpha,\beta}$. For every $\delta>0$ there are numbers $A=A(f,\delta)$ and $B=B(f,\delta)$ such that
	\[\left|\frac{f'(s)}{f(s)}\right| \leq A + \frac{B}{\dist(s,\mathscr{Z}_f)}\]
	for every $s$ in $\mathbb{S}_{\alpha+\delta,\beta-\delta}$. 
\end{lemma}
\begin{proof}
	We let $R_{\delta}(\tau)$ be as in Lemma~\ref{lem:BJ3} and form the function
	\[g(s) = f(s) \prod_{\xi \in \mathscr{Z}_f \cap R_{\delta/2}(\tau)} \frac{1}{s-\xi},\]
	where the zeros are repeated according to multiplicity. By Lemma~\ref{lem:BJ2}, there are at most $N=N(f,\delta)$ factors in this product. Lemma~\ref{lem:BJ3} asserts that $|g(s)| \geq c=c(f,\delta)$ for every $s$ in $[\alpha+\delta/2,\beta-\delta/2]\times[\tau,\tau+1]$. We next use the maximum modulus principle on the rectangle
	\[[\alpha+\delta/8,\beta-\delta/8]\times[\tau-1,\tau+2],\]
	to conclude that $|g(s)| \leq \big(\min(\delta/4,1)\big)^{-N} \|f\|_\infty$ inside the rectangle. This means that the function $\log{g}$ is bounded and analytic in $[\alpha+\delta/2,\beta-\delta/2]\times[\tau,\tau+1]$. Setting $A = \big(c\min(\delta/2,1/4)\big)^{-1}\|g\|_\infty$, we get that $|g'(s)| \leq A c$ in $[\alpha+\delta,\beta-\delta]\times[\tau+1/4,\tau+3/4]$ from Cauchy's integral formula. This means that
	\[\left|\frac{f'(s)}{f(s)} - \sum_{\xi \in \mathscr{Z}_f \cap R_{\delta/2}(\tau)} \frac{1}{s-\xi}\right| \leq A\]
	in $[\alpha+\delta,\beta-\delta]\times[\tau+1/4,\tau+3/4]$, which yields the stated estimate with $B=N$. 
\end{proof}

For later reference, we also record the trivial application of the Cauchy integral formula and of the mean value theorem to functions in $\mathbb{S}_{\alpha,\beta}$.
\begin{lemma}\label{lem:cauchy} 
	If $f$ is an analytic function in $\mathbb{S}_{\alpha,\beta}$ and $\delta>0$, then
	\[|f(s)| \leq \frac{2\|f\|_\infty}{\delta} \dist(s,\mathscr{Z}_f)\]
	for every $s$ in $\mathbb{S}_{\alpha+\delta,\beta-\delta}$. 
\end{lemma}

Let us emphasize that if $f$ is in $\mathscr{H}^\infty$, then the constants appearing in Lemma~\ref{lem:BJ1}, Lemma~\ref{lem:BJ2}, Lemma~\ref{lem:BJ3}, Lemma~\ref{lem:dist} and Lemma~\ref{lem:cauchy} are the same for $f$ and for $f_\chi$. This is simply because the constants are not effected by vertical translations and because $f_\chi$ can be obtained as a uniform limit of vertical translations. 

In preparation for our next result on the zero set of functions in $\mathscr{H}^\infty$, we recall that Jensen's formula (see e.g.~\cite{Titchmarsh1958}*{\S3.61}) states that if $\varphi$ is an analytic function in the unit disc that does not vanish at the origin, then 
\begin{equation}\label{eq:jensen} 
	\sum_{\substack{z \in \varphi^{-1}(\{0\}) \\
	|z| < r}} \log{\frac{r}{|z|}} = \int_0^{2\pi} \log|\varphi(r e^{i \theta})|\,\frac{d\theta}{2\pi} - \log|\varphi(0)| 
\end{equation}
for every $0<r<1$. Let $H^\infty(\mathbb{C}_0)$ stand for the set of bounded analytic functions in $\mathbb{C}_0$ and write $\|f\|_\infty$ for the norm of $H^\infty(\mathbb{C}_0)$. 
\begin{theorem}\label{thm:Hinftyloc} 
	If $f$ is a function in $H^\infty(\mathbb{C}_0)$ enjoying the property that there are constants $\gamma>0$ and $c>0$ such that $|f(\gamma+i \tau)| \geq c > 0$ for every $\tau$ in $\mathbb{R}$, then 
	\begin{align}
		\sum_{\substack{s \in \mathscr{Z}_f \\
		\tau \leq \mim{s} \leq \tau+1 \\
		0 < \mre{s} < \gamma}} \mre{s} &\leq \frac{4\gamma^2+1}{2\gamma} \log{\frac{\|f\|_\infty}{c}} \label{eq:bloc} \intertext{and} \int_\tau^{\tau+1} \big|\log|f(it)|\big|\,dt &\leq \big|\log{\|f\|_\infty}\big| + \frac{\pi}{2} \frac{4\gamma^2+1}{2\gamma} \log{\frac{\|f\|_\infty}{c}} 
	\label{eq:logloc} \end{align}
	for every $\tau$ in $\mathbb{R}$. 
\end{theorem}
\begin{proof}
	If $w$ is in $\mathbb{C}_0$, then the function
	\[\psi_w(s) = \frac{s-w}{s+\overline{w}}\]
	is a conformal map from $\mathbb{C}_0$ to $\mathbb{D}$. Suppose that $f$ in $H^\infty(\mathbb{C}_0)$ does not vanish in the point $w$. Using Jensen's formula \eqref{eq:jensen} on the function $\varphi = f \circ \psi_w^{-1}$ with the estimate $\frac{1}{2}(1-x^2) \leq \log{\frac{1}{x}}$ for $0<x \leq 1$, we get
	\[\sum_{s \in \mathscr{Z}_f} \frac{2\mre{w}\mre{s}}{|s+\overline{w}|^2} \leq \log{\|f\|_\infty}-\log|f(w)|.\]
	
	Fix $\tau$ and choose $w = \gamma+i\tau$. We get an upper bound for the right-hand side from the assumption $|f(w)|\geq c$. We get a lower bound for the left-hand side by restricting to the zeros that satisfy $\tau \leq t \leq \tau+1$ and $0 < \mre{s} < \gamma$, and in this range we have $|s+\overline{w}|^2 \leq 4\gamma^2+1$. This completes the proof of \eqref{eq:bloc}.
	
	For the second estimate, we first use the triangle inequality to estimate
	\[\big|\log|f(it)|\big| \leq \big|\log{\|f\|_\infty}\big| - \log \frac{|f(it)|}{\|f\|_\infty}.\]
	To handle the second term, we infer that
	\[-\int_\tau^{\tau+1}\log \frac{|f(it)|}{\|f\|_\infty}\,dt \leq -\pi \frac{\gamma^2 +1/4}{\gamma} \int_{-\infty}^\infty \frac{\gamma}{\gamma^2+(t-3\tau/2)^2} \log \frac{|f(it)|}{\|f\|_\infty}\,\frac{dt}{\pi}.\]
	Since $\log|f|$ is subharmonic in $\mathbb{C}_0$, we get that
	\[\int_{-\infty}^\infty \frac{\gamma}{\gamma^2+(t-3\tau/2)^2} \log \frac{|f(it)|}{\|f\|_\infty}\,\frac{dt}{\pi} \geq \log \frac{|f(\gamma+i3\tau/2)|}{\|f\|_\infty} \geq \log\frac{c}{\|f\|_\infty}. \qedhere\]
\end{proof}

The estimate \eqref{eq:bloc} from Theorem~\ref{thm:Hinftyloc} can be considered as an improved version of \cite{BP2021}*{Lemma~2.4}. See also \cite{BK2024}*{Theorem~4.1} for an alternative approach through the maximum principle for Green's function. 

Theorem~\ref{thm:Hinftyloc} applies to any nontrivial function from $\mathscr{H}^\infty$. Indeed, suppose that $f$ in $\mathscr{H}^\infty$ enjoys the Dirichlet series expansion \eqref{eq:diriseri} and that $N$ is the smallest positive integer such that $a_N\neq0$. By absolute convergence, there is some $\gamma>0$ such that
\[|f(s)| \geq \frac{|a_N|}{2} N^{-\mre{s}}\]
for $\mre{s}\geq \gamma$. This has the following consequence.
\begin{corollary} \label{cor:suplogint}
    If $f$ is a nontrivial function in $\mathscr{H}^\infty$, then
    \[\sup_{\tau\in\mathbb{R}} \int_\tau^{\tau+1} \big|\log|f(it)|\big|\,dt < \infty.\]
\end{corollary}

The problem of determining the functions $\psi$ on $\mathbb{R}$ such that $\psi(\tau)=|f(i\tau)|$ for some $f$ in $\mathscr{H}^\infty$ was raised in \cite{SS2009}*{Question 1}. Corollary~\ref{cor:suplogint} provides a modicum of progress on this problem, since it illustrates that there are functions that may be realized as boundary values of functions in $H^\infty(\mathbb{C}_0)$, but not of functions from $\mathscr{H}^\infty$. One way to construct a function in $H^\infty(\mathbb{C}_0)$ such that the conclusion of Corollary~\ref{cor:suplogint} does not hold is to form a Blaschke product in any strictly bigger half-plane that has zeros of order $k$ at the points $s=i 2^k$ for $k=1,2,3,\ldots$.

Recall that a Blaschke product in $\mathbb{C}_0$ has the form
\[B(s)=\prod_{\alpha \in \mathscr{Z}_B}\frac{1-\overline{\alpha}^2}{|1-\alpha^2|}\frac{s-\alpha}{s+\overline{\alpha}},\]
where the zero set $\mathscr{Z}_B$ satisfies the condition
\[\sum_{\alpha \in \mathscr{Z}_B}\frac{\mre\alpha}{1+(\mim\alpha)^2}<\infty.\]
We will apply Theorem~\ref{thm:Hinftyloc} in combination with the following two results.
\begin{lemma}\label{lem:blaschkeint} 
	If $B$ is a Blaschke product in $\mathbb{C}_0$, then
	\[\left|\mim{\int_0^\gamma} \frac{B'(\sigma+it)}{B(\sigma+it)}\,\sigma\,d\sigma\right| \leq \left(2\pi + \frac{\pi^2}{3} \gamma^2 \right) \sup_{\tau\in\mathbb{R}} \sum_{\substack{\alpha \in \mathscr{Z}_B \\
	\tau \leq \mim{\alpha} \leq \tau+1 }} \mre{\alpha},\]
	for every $\gamma>0$ and zero-free line $\mim{s}=t$. 
\end{lemma}
\begin{proof}
	We write
	\[\frac{B'(s)}{B(s)} = \sum_{\alpha \in \mathscr{Z}_B} \left(\frac{1}{s-\alpha}-\frac{1}{s+\overline{\alpha}}\right),\]
	and consider each term separately. Let us therefore set
	\[I(t,\alpha) = \mim\int_0^\gamma \left(\frac{1}{\sigma+it-\alpha}-\frac{1}{\sigma+it+\overline{\alpha}}\right)\,\sigma\,d\sigma.\]
	There are two cases to consider. In the first case that $|\mim{\alpha}-t| \leq 1$, we write
	\[\left(\frac{1}{\sigma+it-\alpha}-\frac{1}{\sigma+it+\overline{\alpha}}\right)\sigma = \frac{\alpha-it}{\sigma+it-\alpha} + \frac{\overline{\alpha}+it}{\sigma+it+\overline{\alpha}}\]
	and compute 
	\begin{align*}
		\mim{\frac{\alpha-it}{\sigma+it-\alpha}} &= \frac{\sigma(\mim{\alpha}-t)}{(\sigma-\mre{\alpha})^2 + (\mim{\alpha}-t)^2}, \\
		\mim{\frac{\overline{\alpha}+it}{\sigma+it+\overline{\alpha}}} &= \frac{\sigma(t-\mim{\alpha})}{(\sigma+\mre{\alpha})^2 + (\mim{\alpha}-t)^2}. 
	\end{align*}
	This shows that 
	\begin{multline*}
		I(t,\alpha) = (\mim{\alpha}-t) \times \\
		\int_0^\gamma \left(\frac{\sigma}{(\sigma-\mre{\alpha})^2 + (\mim{\alpha}-t)^2} - \frac{\sigma}{(\sigma+\mre{\alpha})^2 + (\mim{\alpha}-t)^2}\right)\,d\sigma. 
	\end{multline*}
	The integrand is nonnegative, so by letting $\gamma \to \infty$ we get that $|I(t,\alpha)| \leq \pi \mre{\alpha}$. In the second case that $k < |\mim{\alpha}-t| \leq {k+1}$ for $k=1,2,3,\ldots$, we estimate crudely
	\[\left|\left(\frac{1}{\sigma+it-\alpha}-\frac{1}{\sigma+it+\overline{\alpha}}\right)\right| = \left|\frac{2\mre{\alpha}}{(\sigma+it-\alpha)(\sigma+it+\overline{\alpha})}\right| \leq \frac{2\mre{\alpha}}{k^2}\]
	to infer that $|I(\alpha,t)| \leq \frac{\gamma^2}{k^2} \mre{\alpha}$. In light of the dominated convergence theorem we can add up the contribution for each $k$ to obtain the stated estimate. 
\end{proof}
\begin{lemma}\label{lem:outer} 
	Suppose that $f$ is a nontrivial bounded analytic function in $\mathbb{C}_0$ and let $F$ denote the outer part of $f$. Then
	\[\left|\mim\int_0^\gamma \frac{F'(\sigma+it)}{F(\sigma+it)} \,\sigma\,d\sigma \right| \leq \left(1+\frac{\pi}{6}\gamma \right) \sup_{t \in \mathbb{R}} \int_t^{t+1} \big|\log|f(i \tau)|\big| \,d\tau \]
	for every $\gamma>0$. 
\end{lemma}
\begin{proof}
	Since
	\[\log|F(\sigma+it)| = \int_{-\infty}^\infty \frac{\sigma}{\sigma^2+(t-\tau)^2}\,\log|f(i\tau)| \,\frac{d\tau}{\pi}\]
	we get that
	\[\mim \frac{F'(\sigma+it)}{F(\sigma+it)} = -\int_{-\infty}^\infty \frac{2\sigma (t-\tau)}{(\sigma^2+(t-\tau)^2)^2}\,\log|f(i\tau)| \,\frac{d\tau}{\pi}\]
	from the Cauchy--Riemann equations. We get from the triangle inequality and Tonelli's theorem that
	\[\left|\mim\int_0^\gamma \frac{F'(\sigma+it)}{F(\sigma+it)} \,\sigma\,d\sigma \right| \leq \int_{-\infty}^\infty \int_0^\gamma \frac{2\sigma^2 |t-\tau|}{(\sigma^2+(t-\tau)^2)^2}\, d\sigma\, \big|\log|f(i\tau)|\big|\,\frac{d\tau}{\pi}.\]
	If $|t-\tau|\leq 1$, then we estimate
	\[\int_0^\gamma \frac{2\sigma^2 |t-\tau|}{(\sigma^2+(t-\tau)^2)^2 }\, d\sigma \leq \int_0^\infty \frac{2\sigma^2 |t-\tau|}{(\sigma^2+(t-\tau)^2)^2}\, \,d\sigma = \frac{\pi}{2}.\]
	If $k<|t-\tau|\leq k+1$ for some $k=1,2,3,\ldots$, then
	\[\int_0^\gamma \frac{2\sigma^2 |t-\tau|}{(\sigma^2+(t-\tau)^2)^2 }\, d\sigma \leq \frac{1}{2 k^2} \int_0^\gamma 1\,d\sigma = \frac{\gamma}{2k^2},\]
	which completes the proof.
\end{proof}

\section{The Hardy--Stein identity and the Littlewood--Paley formula} \label{sec:LP} 
Following \cite{Stein1933}, we will establish Theorem~\ref{thm:HS} via Green's theorem. Consider a rectangle $R = [\sigma_0,\sigma_1]\times[-T,T]$. If $u$ and $v$ are continuously differentiable functions on $R$, then
\[\iint_R \left(\frac{\partial}{\partial \sigma} u - \frac{\partial}{\partial t} v\right)\,dtd\sigma = \oint_{\partial R} \big(u \,dt + v\,d\sigma\big)\]
where $\partial R$ is oriented counter-clockwise. If $f$ is analytic and does not vanish in $R$, then $|f|^p$ is smooth on $R$ and we could choose $u = -\frac{\partial}{\partial \sigma} |f|^p$ and $v = \frac{\partial}{\partial t} |f|^p$ so that
\[\frac{\partial}{\partial \sigma} u - \frac{\partial}{\partial t} v = -\Delta |f|^p = -p^2 |f|^{p-2} |f'|^2.\]
We could then obtain the stated formula by estimating the integrals on the segment $[\sigma_1-iT,\sigma_1+iT]$ and on the segments $[\sigma_0\pm iT,\sigma_1\pm iT]$ before letting $\sigma_1 \to \infty$, dividing by $2T$, and finally letting $T \to \infty$. To make this argument work for a nontrivial function in $\mathscr{H}^\infty$ that has zeros, we need to make certain adjustments.

The setup is as follows. Fix $1 \leq p<\infty$ and $\sigma_0>0$. We want to find a sequence of functions $(\varphi_j)_{j\geq1}$ enjoying the following properties. 
\begin{enumerate}
	\item[\normalfont (i)] For every $j\geq1$, the function $\varphi_j$ is almost periodic in $\mathbb{C}_{\sigma_0}$ and $\frac{\partial}{\partial \sigma} \varphi_j$ is uniformly continuous in $\mathbb{C}_{\sigma_0}$. 
	\item[\normalfont (ii)] The sequence $(\varphi_j)_{j\geq1}$ converges uniformly to $|f|^p$ in $\mathbb{C}_{\sigma_0}$. 
	\item[\normalfont (iii)] It holds that
	\[\lim_{j\to \infty} \limsup_{T \to \infty} \sup_{\kappa>\sigma_0} \left|\frac{1}{2T} \int_{-T}^T \frac{\partial}{\partial \kappa} \varphi_j(\kappa+it)\,dt +\frac{1}{2T} \int_{\kappa}^\infty \int_{-T}^T \Delta |f(\sigma+it)|^p \,dt d\sigma\right|=0.\]
\end{enumerate}

Before constructing the sequence $(\varphi_j)_{j\geq1}$ and demonstrating that it satisfies the requirements (i)--(iii), let us explain how this would lead to a proof of Theorem~\ref{thm:HS}. The first assertion in the requirement (i) ensures that the mean value
\[\Phi_j(\kappa) = \lim_{T\to\infty} \frac{1}{2T} \int_{-T}^T \varphi_j(\kappa+it)\,dt\]
exists for every $\kappa>\sigma_0$. The second assertion in (i) ensures that $\frac{\partial}{\partial \sigma} \varphi_j$ is almost periodic (see e.g.~\cite{Besicovitch1955}*{p.~6}) and that
\[\frac{\partial}{\partial \kappa} \Phi_j(\kappa) = \lim_{T\to\infty} \frac{1}{2T} \int_{-T}^T \frac{\partial}{\partial \kappa}\varphi_j(\kappa+it)\,dt.\]
The requirement (ii) shows that $\Phi_j(\kappa)$ converges uniformly to $M_p^p(f,\kappa)$ in $\mathbb{C}_0$ as $j\to \infty$. The statement of Theorem~\ref{thm:HS} follows from this and (iii).

Our final preparation for the proof of Theorem~\ref{thm:HS} is the following result, which is essentially due to Hardy \cite{Hardy1915}. The proof is a direct calculation based on the Cauchy--Riemann equations (see \cite{BP2021}*{Lemma~3.4}).
\begin{lemma}\label{lem:hardy} 
	Suppose that $f$ is an analytic function in some domain $\Omega$ and that $\psi \colon \mathbb{R} \to \mathbb{R}$ is a $C^2$ function such that $\supp{\psi'} \cap (-\infty,0]$ is compact. For every $s$ in $\Omega$ it holds that 
	\begin{align*}
		\frac{\partial}{\partial s} \psi(\log{|f(s)|}) &= \psi'(\log|f(s)|) \frac{f'(s)}{f(s)}, \\
		\Delta \psi(\log{|f(s)|}) &= \psi''(\log|f(s)|) \left|\frac{f'(s)}{f(s)}\right|^2, 
	\end{align*}
	where $\frac{\partial}{\partial s} = \frac{\partial}{\partial \sigma} - i \frac{\partial}{\partial t}$ and $\Delta = \frac{\partial^2}{\partial \sigma^2} + \frac{\partial^2}{\partial t^2}$. 
\end{lemma}

Let us now proceed with the construction.
\begin{proof}
	[Proof of Theorem~\ref{thm:HS}] Fix $1 \leq p<\infty$ and let $\psi_0 \colon \mathbb{R} \to \mathbb{R}$ be an increasing $C^2$ function satisfying $\psi_0(x) = \exp(x)$ for $x \geq 0$ and $\psi_0(x) = 0$ for $x \leq -1$. For $j=1,2,3,\ldots$, we define
	\[\varphi_j(s) = j^{-1} \psi_0\left(\log{j}+p\log|f(s)|\right).\]
	The sequence $(\varphi_j)_{j\geq1}$ enjoys the following properties: 
	\begin{itemize}
		\item If $|f(s)|^p \geq j^{-1}$, then $\varphi_j(s) = |f(s)|^p$. 
		\item If $|f(s)|^p \leq (ej)^{-1}$, then $\varphi_j(s) =0$. 
		\item If $(ej)^{-1} \leq |f(s)|^p \leq j^{-1}$, then we apply Lemma~\ref{lem:hardy} to the function $\psi_j(x) = j^{-1}\psi_0(\log{j}+px)$ and get that
		\[\left|\frac{\partial}{\partial s}\varphi_j(s)\right| \leq C_1 |f(s)|^{p-1} |f'(s)| \qquad \text{and}\qquad |\Delta \varphi_j(s)| \leq C_2 |f(s)|^{p-2} |f'(s)|^2,\]
		where $\displaystyle C_1 = pe \max_{x \in [-1,0]} |\psi_0'(x)|$ and $\displaystyle C_2 = p^2 e \max_{x \in [-1,0]} |\psi_0''(x)|$. 
	\end{itemize}
	It is plain that (i) holds, where we in the first assertion use that each $\psi_j$ is uniformly continuous on any interval $(-\infty,b]$ to deduce the almost periodicity of $\varphi_j$ from that of $f$ (see \cite{Besicovitch1955}*{p.~3}). It is also clear that (ii) holds. To handle (iii), we begin by using Green's theorem as outlined above to get 
	\begin{align*}
		-\int_\kappa^{\sigma_1} \int_{-T}^T \Delta \varphi_j(s)\,dtd\sigma &= \int_{-T}^T \frac{\partial}{\partial \kappa} \varphi_j(\kappa+it)\,dt + \int_{\kappa}^{\sigma_1} \frac{\partial}{\partial t} \varphi_j(\sigma-iT)\,d\sigma \\
		&-\int_{-T}^T \frac{\partial}{\partial \sigma_1} \varphi_j(\sigma_1+it)\,dt - \int_{\kappa}^{\sigma_1} \frac{\partial}{\partial t} \varphi_j(\sigma+iT)\,d\sigma. 
	\end{align*}
	Let $I_2$, $I_3$, and $I_4$ denote the three latter integrals on the right-hand side. Our first task is to estimate their contribution as $\sigma_1 \to \infty$. It is possible to use the results from Section~\ref{sec:zeroset}, but we can we can get away with rather crude estimates. The point is that there is a constant $C_3 = C_3(f,\sigma_0)$ such that $|f'(\sigma+it)| \leq C_3 2^{-\sigma}$ for $\sigma \geq \sigma_0$, since $f$ is in $\mathscr{H}^\infty$. We get
	\[\left|\frac{\partial}{\partial s}\varphi_j(s)\right| \leq C_4 2^{-\sigma} \qquad \text{for} \qquad C_4 = \max(C_1,1) \|f\|_{\mathscr{H}^\infty}^{p-1} C_3\]
	for $\mre{s} \geq \sigma_0$. Since $\varphi_j$ is real-valued, we can estimate $\left|\frac{\partial}{\partial \sigma} \varphi_j \right|$ and $\left|\frac{\partial}{\partial t} \varphi_j\right |$ from above by $\left|\frac{\partial}{\partial s} \varphi_j\right|$ to get
	\[\left|I_2 + I_3 + I_4\right| \leq C_4 \frac{2}{\log{2}} \left(2^{-\kappa} + (T\log{2}-1) 2^{-\sigma_1}\right).\]
	If we let $\sigma_1 \to \infty$, we get a contribution of at most $C_4 \frac{2}{\log{2}} 2^{-\sigma_0}$ independently of $\kappa\geq \sigma_0$ and $T>0$. In view of this and Green's theorem, we see that the requirement (iii) holds if we can prove that
	\[\lim_{j\to \infty} \limsup_{T \to \infty} \frac{1}{2T} \int_{\sigma_0}^\infty \int_{-T}^T \big|\Delta \varphi_j(\sigma+it) - \Delta |f(\sigma+it)|^p\big| \,dt d\sigma=0.\]
	We will split the integral over $\sigma$ in two parts. Since $f$ is a nontrivial Dirichlet series, there is some $\gamma>0$ such that $f$ does not vanish in $\overline{\mathbb{C}_\gamma}$. To handle the first part, we fix $\delta>0$ sufficiently small and use Lemma~\ref{lem:BJ1} to infer that there is a constant $m>0$ such that $|f(s)| \geq m >0$ whenever $\dist(s,\mathscr{Z}(f)) \geq \delta$. We restrict our attention to $j \geq 1/m$. By Lemma~\ref{lem:BJ2} there are at most $N 2(T+1)$ zeros of $f$ in the rectangle $[\sigma_0,\gamma]\times [-T,T]$ and so $|f(s)| \leq 1/j$ occurs in at most $N 2(T+1)$ discs of radius $\delta$ centered at the zeros of $f$. Outside of these discs, we have $\Delta \varphi_j = \Delta |f|^p$. At the points inside this disc where $f$ does not vanish, we have
	\[\big|\Delta \varphi_j(\sigma+it) - \Delta |f(\sigma+it)|^p\big| \leq C_5 |f(s)|^{p-2} |f'(s)|^2 \qquad \text{for} \qquad C_5 = C_2 + p^2.\]
	Using Lemma~\ref{lem:dist} and Lemma~\ref{lem:cauchy}, we estimate
	\[|f(s)|^{p-2} |f'(s)|^2 \leq 2 \|f\|_{\mathscr{H}^\infty}^p \left(\frac{2}{\sigma_0}\right)^p\left(A^2\big(\dist(s,\mathscr{Z}_f)\big)^p + B^2 \big(\dist(s,\mathscr{Z}_f)\big)^{p-2}\right),\]
	where $A$ and $B$ only depend on $f$, $\sigma_0$, and $\gamma$. Thus 
	\begin{multline*}
		\lim_{j\to\infty}\limsup_{T\to\infty} \frac{1}{2T} \int_{\sigma_0}^\gamma \int_{-T}^T \big|\Delta \varphi_j(\sigma+it) - \Delta |f(\sigma+it)|^p\big| \,dt d\sigma \\
		\leq \inf_{\delta > 0} 2^{2p+3}\sigma_0^{-p}C_5 \|f\|_{\mathscr{H}^\infty}^p N \left(A^2 \delta^{p+2} + B^2 2 \delta^p \right) = 0. 
	\end{multline*}
	The contribution from $\gamma \leq \sigma < \infty$ is much easier to handle. If $f(+\infty) \neq 0$, then $f$ is bounded below on $\overline{\mathbb{C}_\gamma}$ so for all sufficiently large $j$ we have $\Delta \varphi_j = \Delta |f|^p$ throughout. If $f(+\infty)=0$, then straightforward estimates using the exponential decay of $f$ and $f'$ show that the contribution from $\gamma \leq \sigma < \infty$ can be bounded by $C_6/j$. 
\end{proof}
\begin{proof}
	[Proof of Corollary~\ref{cor:LP}] We start from \eqref{eq:diffint} and use that the limit in Theorem~\ref{thm:HS} is uniform to move the limit outside the integral to the effect that
	\[M_p^p (f,\sigma_0) = M_p^p(f,\sigma_1) + \lim_{T\to \infty} \frac{p^2}{2T} \int_{\sigma_0}^{\sigma_1} \int_\kappa^\infty \int_{-T}^T |f(s)|^{p-2} |f'(s)|^2 \, dt d\sigma d\kappa.\]
	For fixed $T>0$, we split the integral over $\sigma$ in two parts. For the first part, we use Tonelli's theorem to get
	\[\int_{\sigma_0}^{\sigma_1} \int_\kappa^{\sigma_1} \int_{-T}^T |f(s)|^{p-2} |f'(s)|^2 \, dt d\sigma d\kappa = \int_{\sigma_0}^{\sigma_1} \int_{-T}^T |f(s)|^{p-2} |f'(s)|^2 \,(\sigma-\sigma_0)\, dt d\sigma.\]
	For the second part, we get
	\[\int_{\sigma_0}^{\sigma_1} \int_{\sigma_1}^{\infty} \int_{-T}^T |f(s)|^{p-2} |f'(s)|^2 \, dt d\sigma d\kappa = (\sigma_1-\sigma_0) \int_{\sigma_1}^{\infty} \int_{-T}^T |f(s)|^{p-2} |f'(s)|^2 \, dt d\sigma.\]
	This leads to 
	\begin{multline*}
		M_p^p (f,\sigma_0) = M_p^p(f,\sigma_1) + \lim_{T\to\infty} \frac{p^2}{2T} \int_{\sigma_0}^{\sigma_1} \int_{-T}^T |f(s)|^{p-2} |f'(s)|^2 \,(\sigma-\sigma_0)\,dtd\sigma \\
		- (\sigma_1-\sigma_0 )\frac{\partial}{\partial \sigma_1} M_p^p(f,\sigma_1) 
	\end{multline*}
	by way of Theorem~\ref{thm:HS} again. It is plain that the final term decays exponentially as $\sigma_1 \to \infty$, so we obtain the stated result by letting $\sigma_1 \to \infty$ and then $\sigma_0 \to 0^+$. 
\end{proof}
\begin{proof}
	[Proof of Theorem~\ref{thm:LPswap}] We follow the proof of Theorem~\ref{thm:HS}, but fix $T>0$ and use
	\[u = \varphi_j-(\sigma-\sigma_0)\frac{\partial}{\partial \sigma} \varphi_j \qquad \text{and} \qquad v = (\sigma-\sigma_0) \frac{\partial}{\partial t} \varphi_j.\]
	We let $\sigma_0 \to 0^+$ and $\sigma_1\to \infty$, before taking the limit $j \to \infty$ using the dominated convergence theorem to obtain 
	\begin{multline*}
		\left|\int_{-T}^T |f(i\tau)|^p \,d\tau - |f(+\infty)|^p - p^2\int_{-T}^T \int_0^\infty |f(s)|^{p-2} |f'(s)|^2 \sigma\,d\sigma dt\right| \\
		\leq I_f(T) + I_f(-T), 
	\end{multline*}
	where
	\[I_f(\tau) = \int_0^\infty |f(\sigma+i\tau)|^{p-1} |f'(\sigma+i\tau)|\,\sigma d\sigma.\]
	It is sufficient to consider the integral $(0,\gamma)$ due to the exponential decay of
	\[\sigma \mapsto |f(\sigma+i\tau)|^{p-1} |f'(\sigma+i\tau)|\]
	as $\sigma \to \infty$. We can then use the trivial estimate $|f'(\sigma+i\tau)| \leq \|f\|_{\mathscr{H}^\infty}/\sigma$ on $(0,\gamma)$ to conclude that $I_f(\tau) \leq C$ for $C=C(f)$, which wraps up the proof. 
\end{proof}
\begin{theorem}\label{thm:LPchiswap} 
	Fix $1 \leq p<\infty$. If $f$ is in $\mathscr{H}^\infty$, then 
	\begin{align*}
		\|f\|_{\mathscr{H}^p}^p &= |f(+\infty)|^p + p^2 \int_{\mathbb{T}^\infty} \int_0^\infty |f_\chi(\sigma)|^{p-2} |f_\chi'(\sigma)|^2 \,\sigma d\sigma dm_\infty(\chi) \intertext{and} \|f\|_{\mathscr{H}^p}^p &= |f(+\infty)|^p + \lim_{T\to\infty} \frac{p^2}{2T} \int_{-T}^T \int_0^\infty |f_\chi(s)|^{p-2} |f_\chi'(s)|^2 \,\sigma d\sigma dt 
	\end{align*}
	for almost every $\chi$ on $\mathbb{T}^\infty$. 
\end{theorem}
\begin{proof}
	Since $\|f_\chi\|_{\mathscr{H}^p} = \|f\|_{\mathscr{H}^p}$ and $f_\chi(+\infty)=f(+\infty)$ for every $\chi$ on $\mathbb{T}^\infty$, we get
	\[\|f\|_{\mathscr{H}^p}^p = |f(+\infty)|^p + \lim_{\sigma_0 \to 0^+} \lim_{T\to\infty} \frac{p^2}{2T} \int_{\sigma_0}^\infty \int_{-T}^T |f_\chi(s)|^{p-2} |f_\chi'(s)|^2 \,(\sigma-\sigma_0)\,dtd\sigma\]
	from Corollary~\ref{cor:LP}. We now wish to integrate this quantity over $\mathbb{T}^\infty$ and move the limits outside the integral (preserving their order). The $\sigma_0$-limit can be moved outside by the monotone convergence theorem. The limit in $T$ can be commuted through after noting that the convergence is uniform in $\chi$ (with a constant depending on $\sigma_0>0$), since the estimates used in the proof of Theorem~\ref{thm:HS} are all independent of $\chi$, cf. the remark following Lemma~\ref{lem:dist}. Using Tonelli's theorem and the rotation invariance of the Haar measure $m_\infty$ of $\mathbb{T}^\infty$, we obtain
	\[\|f\|_{\mathscr{H}^p}^p = |f(+\infty)|^p + \lim_{\sigma_0 \to 0^+} p^2 \int_{\sigma_0}^\infty \int_{\mathbb{T}^\infty} |f_\chi(\sigma)|^{p-2} |f_\chi'(\sigma)|^2\, dm_\infty(\chi) \,(\sigma-\sigma_0)\,d\sigma.\]
	Using Tonelli's theorem again and then the monotone convergence theorem, we obtain the first-stated formula. Since $\|f\|_{\mathscr{H}^p} \leq \|f\|_{\mathscr{H}^\infty} < \infty$, the first-stated formula implies that the function
	\[\chi \mapsto p^2 \int_0^\infty |f_\chi(\sigma)|^{p-2} |f_\chi'(\sigma)|^2 \,\sigma d\sigma\]
	is in $L^1(\mathbb{T}^\infty)$. The second-stated formula follows from this and the second version of the ergodic theorem discussed in Section~\ref{sec:bvbv}. 
\end{proof}

\section{Jensen's formula and the mean counting function} \label{sec:jensen} 
The first goal of the present section is to implement a version of Jensen's formula \eqref{eq:jensen} in the theory of $\mathscr{H}^\infty$ and explain how it is related to the mean counting function. We begin with the following result.
\begin{theorem}[Jessen \cite{Jessen1933}] \label{thm:jessen} 
	Suppose that $f$ is a nontrivial analytic almost periodic function in the strip $\mathbb{S}_{\alpha,\beta}$. The limit
	\[\mathscr{J}_f(\sigma) = \lim_{T\to\infty} \frac{1}{2T} \int_{-T}^T \log|f(\sigma+it)|\,dt\]
	exists for every $\alpha<\sigma<\beta$ and defines a convex function of $\sigma$. If $(f_j)_{j\geq1}$ is a sequence of analytic almost periodic functions in $\mathbb{S}_{\alpha,\beta}$ converging uniformly to $f$, then
	\[\lim_{j\to\infty} \mathscr{J}_{f_j}(\sigma) = \mathscr{J}_f(\sigma)\]
	for each fixed $\alpha<\sigma<\beta$. 
\end{theorem}

A streamlined version of Jessen's argument can be found in \cite{Levin1980}*{Section~VI.3}. It is also possible to establish the result by techniques similar to those used in the proof of the Hardy--Stein formula above (compare with \cite{BP2021}*{Section~3}). 

If we apply Theorem~\ref{thm:jessen} to a nontrivial function $f$ in $\mathscr{H}^\infty$, then its behavior as $\mre{s} \to \infty$ and the convexity of $\mathscr{J}_f$ ensures that $\mathscr{J}_f$ is non-decreasing. Moreover, Lemma~\ref{lem:vlf} and the final assertion of Theorem~\ref{thm:jessen} ensure that
\[\mathscr{J}_f(\sigma) = \mathscr{J}_{f_\chi}(\sigma)\]
for every $\chi$ on $\mathbb{T}^\infty$ and every $\sigma>0$. Combining these two assertions yields the following result (see \cite{BP2021}*{Theorem~3.7} for the proof), that will be used in the proof of Theorem~\ref{thm:largelim}.
\begin{lemma}\label{lem:jessenchi} 
	If $f$ is a nontrivial function in $\mathscr{H}^\infty$, then $\log|f^\ast|$ is in $L^1(\mathbb{T}^\infty)$ and
	\[\lim_{\sigma \to 0^+} \mathscr{J}_f(\sigma) \leq \int_{\mathbb{T}^\infty} \log|f^\ast(\chi)|\,dm_\infty(\chi).\]
\end{lemma}

Our approach to Jensen's formula for $\mathscr{H}^\infty$ is based on Littlewood's argument principle \cite{Littlewood1924} (see also \cite{Titchmarsh1958}*{\S3.8}), which we shall now recall. Let $\Omega$ be a domain in the complex plane which contains the rectangle $R=[\sigma_0,\sigma_1]\times[-T,T]$. If $f$ is analytic in $\Omega$ and $f$ does not vanish on the segments $[\sigma_0-iT,\sigma_1-iT]$, $[\sigma_1-iT,\sigma_1+iT]$, and $[\sigma_0+iT,\sigma_1+iT]$, then 
\begin{equation}\label{eq:littlewood} 
	\begin{split}
		2\pi \sum_{\substack{s \in f^{-1}(\{0\}) \\
		|\mim{s}|<T \\
		\sigma_0 < \mre{s} < \sigma_1}} \left(\mre{s}-\sigma_0\right) &= \int_{-T}^T \log|f(\sigma_0+it)|\,dt + \int_{\sigma_0}^{\sigma_1} \arg{f(\sigma+iT)}\,d\sigma \\
		&-\int_{-T}^T \log|f(\sigma_1+it)|\,dt-\int_{\sigma_0}^{\sigma_1} \arg{f(\sigma-iT)}\,d\sigma. 
	\end{split}
\end{equation}
Here $\arg{f}$ denotes a continuous branch of the argument in a simply connected domain that contains $[\sigma_0-iT,\sigma_1-iT]\cup [\sigma_1-iT,\sigma_1+iT]\cup[\sigma_0+iT,\sigma_1+iT]$.

The contribution of the horizontal integrals in \eqref{eq:littlewood} can be controlled on a relatively dense set $(\pm T_j)_{j\geq1}$ using the almost periodicity of $f$. This can be extended to the general case via Lemma~\ref{lem:BJ2}, yielding the following result. It is identical to \cite{BP2021}*{Lemma~6.1}, except that we have removed the unnecessary assumption that $f$ does not vanish on the line $\mre{s}=\sigma_0$ from the statement.
\begin{theorem}[Jensen's formula] \label{thm:jensen} 
	If $f$ is in $\mathscr{H}^\infty$ and $f(+\infty)\neq0$, then
	\[\lim_{T\to\infty} \frac{\pi}{T} \sum_{\substack{s \in f^{-1}(\{0\}) \\
	|\mim{s}|<T \\
	\sigma_0 < \mre{s} < \infty}} \left(\mre{s}-\sigma_0\right) = \mathscr{J}_f(\sigma_0) - \log|f(+\infty)|\]
	for every $\sigma_0>0$. 
\end{theorem}

The existence of the mean counting function \eqref{eq:meancounting} follows at once from Lemma~\ref{lem:jessenchi} and Theorem~\ref{thm:jensen}, applied to the Frostman shifts $f_\xi$ for $\xi$ in $\mathbb{D}$, as introduced in \eqref{eq:frostman}. It is an easy consequence of Bohr's theorem that if $f$ is in the unit ball of $\mathscr{H}^\infty$, then so is $f_\xi$ for every $\xi$. Note that this argument also recovers the Littlewood-type estimate \eqref{eq:Lest} since $|f_\xi^\ast| \leq 1$ almost everywhere on $\mathbb{T}^\infty$. 
\begin{proof}
	[Proof of Theorem~\ref{thm:Nfswap}] If $f$ is a Dirichlet series mapping $\mathbb{C}_0$ to $\mathbb{D}$, we let $f_\xi = B_\xi S_\xi F_\xi$ denote the canonical factorization of the Frostman shift $f_\xi$ viewed as a function in $H^\infty(\mathbb{C}_0)$. We restrict our attention to the quasi-every $\xi$ in $\mathbb{D} \setminus\{f(+\infty)\}$ for which $f_\xi=B_\xi F_\xi$. Since $f_\xi(s) = 0$ if and only if $f(s)=\xi$, we have 
	\begin{equation}\label{eq:Nfxi} 
		N_f(\xi,T) = \frac{\pi}{T} \sum_{\substack{s \in f_\xi^{-1}(\{0\}) \\
		|\mim{s}| < T \\
		0<\mre{s}<\infty}}\mre{s}. 
	\end{equation}
	If $\xi \neq f(+\infty)$, then there is $\gamma=\gamma(f,\xi)>0$ such that $f_\xi$ does not vanish in $\overline{\mathbb{C}_\gamma}$. This means we can restrict the sum in \eqref{eq:Nfxi} to $0<\mre{s}<\gamma$. We will apply Littlewood's argument principle \eqref{eq:littlewood} to the function $f_\xi$ on the rectangle $[\sigma_0,\gamma] \times [-T,T]$, for $T$ such that $f_\xi$ does not vanish for $\mim{s}= \pm T$. We use integration by parts to write 
	\begin{align*}
		\int_{\sigma_0}^{\gamma} \arg{f_\xi(\sigma \pm iT)}\,d\sigma = (\gamma-\sigma_0) \arg{f_\xi(\gamma+iT)} &-\mim \int_{\sigma_0}^\gamma \frac{B_\xi'(\sigma \pm iT)}{B_\xi(\sigma\pm iT)}\,(\sigma-\sigma_0)\,d\sigma\\
		&-\mim \int_{\sigma_0}^\gamma \frac{F_\xi'(\sigma \pm iT)}{F_\xi(\sigma\pm iT)}\,(\sigma-\sigma_0)\,d\sigma. 
	\end{align*}
	Using Lemma~\ref{lem:blaschkeint} and Lemma~\ref{lem:outer} with Theorem~\ref{thm:Hinftyloc}, we infer from this that
	\[\lim_{\sigma_0 \to 0^+}\left|\int_{\sigma_0}^{\gamma} \arg{f_\xi(\sigma \pm iT)}\,d\sigma\right| \leq C(f_\xi,\gamma).\]
	Since $f_\xi$ does not have a singular inner factor, we know that
	\[\lim_{\sigma_0 \to 0^+} \int_{-T}^T \log|f_\xi(\sigma_0+it)|\,dt = \int_{-T}^T \log|f_\xi(i\tau)|\,d\tau\]
	for every fixed $T>0$. Dividing by $2T$ and letting $T\to\infty$, we obtain that
	\[\lim_{T\to\infty} \left|N_f(\xi,T) - \frac{1}{2T} \int_{-T}^T \log|f_\xi(i\tau)|\,d\tau + \frac{1}{2T}\int_{-T}^T \log|f_\xi(\gamma+it)|\,dt \right|=0,\]
	at first when the limit is taken over $T$ for which $f_\xi$ has no zeros on $\mim{s} = T$. However, armed with Theorem~\ref{thm:Hinftyloc} this is immediately extended to hold for the full limit. The proof is completed by noting that, since $f_\xi$ does not vanish in $\mathbb{C}_\gamma$, 
	\begin{equation}\label{eq:firstcoeff} 
		\lim_{T \to \infty} \frac{1}{2T}\int_{-T}^T \log|f_\xi(\gamma+it)|\,dt = \log|f_\xi(+\infty)|, 
	\end{equation}
	as seen for example from Theorem~\ref{thm:jensen}. 
\end{proof}

We can now establish the following result, which was also mentioned in the introduction.
\begin{theorem}\label{thm:limsup} 
	If $f$ is a Dirichlet series mapping $\mathbb{C}_0$ to $\mathbb{D}$, then
	\[\limsup_{T\to \infty} N_f(\xi,T) \leq \log\left|\frac{1-\overline{\xi}f(+\infty)}{\xi-f(+\infty)}\right|\]
	for every $\xi$ in $\mathbb{D} \setminus \{f(+\infty)\}$. 
\end{theorem}
\begin{proof}
	Since $f_\xi(s)=0$ if and only if $B_\xi(s)=0$, we can repeat the argument in the proof of Theorem~\ref{thm:Nfswap} to conclude that
	\[\lim_{T\to\infty} \left|N_f(\xi,T) + \frac{1}{2T} \int_{-T}^T \log|B_\xi(\gamma+it)|\,dt\right| =0\]
	where $\gamma>0$ is such that $f$ does not take the value $\xi$ in $\overline{\mathbb{C}_\gamma}$. Since $|B_\xi| \geq |f_\xi|$, we conclude by \eqref{eq:firstcoeff} that
	\[ \limsup_{T \to \infty} N_f(\xi,T) \leq -\liminf_{T\to\infty} \frac{1}{2T} \int_{-T}^T \log|f_\xi(\gamma+it)| = -\log|f_\xi(+\infty)|. \qedhere \]
\end{proof}

\section{Proof of Theorem~\ref{thm:largelim} and Theorem~\ref{thm:nolim}} \label{sec:proofproof} 
Via the Kronecker flow, the imaginary line embeds into $\mathbb{T}^2$ as the set
\[\mathscr{L} = \left\{(2^{-i\tau},3^{-i\tau}) \,:\, \tau \in \mathbb{R}\right\}.\]
It is useful to identify $\mathbb{T}^2$ with $\mathbb{R}/ 2\pi \times \mathbb{R}/2\pi$. In $\mathbb{R}/2\pi \times \mathbb{R}/2\pi$ we identify $\mathscr{L}$ with a countable set of line segments that are parallel to the line $(\log{2})y = (\log{3})x$. See Figure~\ref{fig:L}. In particular, $\mathscr{L}$ is a dense set of measure $0$ on $\mathbb{T}^2$.
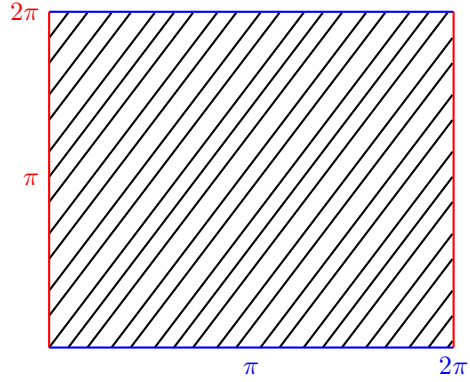
\begin{figure}
	\centering 
	\begin{tikzpicture}
		\begin{axis}
			[ xmin = -1, ymin = -1, hide axis, axis line style=thin, xtick={0,3.1415,6.2832}, xticklabels={$0$,$\pi$,$2\pi$}, ytick={0,3.1415,6.2832}, yticklabels={$0$,$\pi$,$2\pi$}, ]
			
			\addplot [color=black, thick] coordinates {(0,0) (3.964248557502618,6.283185307179586)}; \addplot [color=black, thick] coordinates {(3.964248557502618,0) (6.283185307179586,3.6754277897821974)}; \addplot [color=black, thick] coordinates {(0,3.6754277897821974) (1.645311807825649,6.283185307179586)}; \addplot [color=black, thick] coordinates {(1.645311807825649,0) (5.609560365328268,6.283185307179586)}; \addplot [color=black, thick] coordinates {(5.609560365328268,0) (6.283185307179586,1.0676702723848086)}; \addplot [color=black, thick] coordinates {(0,1.0676702723848086) (3.290623615651298,6.283185307179586)}; \addplot [color=black, thick] coordinates {(3.290623615651298,0) (6.283185307179586,4.743098062167007)}; \addplot [color=black, thick] coordinates {(0,4.743098062167007) (0.9716868659743294,6.283185307179586)}; \addplot [color=black, thick] coordinates {(0.9716868659743294,0) (4.93593542347695,6.283185307179586)}; \addplot [color=black, thick] coordinates {(4.93593542347695,0) (6.283185307179586,2.135340544769617)}; \addplot [color=black, thick] coordinates {(0,2.135340544769617) (2.6169986737999813,6.283185307179586)}; \addplot [color=black, thick] coordinates {(2.6169986737999813,0) (6.283185307179586,5.810768334551819)}; \addplot [color=black, thick] coordinates {(0,5.810768334551819) (0.29806192412300986,6.283185307179586)}; \addplot [color=black, thick] coordinates {(0.29806192412300986,0) (4.26231048162563,6.283185307179586)}; \addplot [color=black, thick] coordinates {(4.26231048162563,0) (6.283185307179586,3.2030108171544285)}; \addplot [color=black, thick] coordinates {(0,3.2030108171544285) (1.9433737319486588,6.283185307179586)}; \addplot [color=black, thick] coordinates {(1.9433737319486588,0) (5.90762228945128,6.283185307179586)}; \addplot [color=black, thick] coordinates {(5.90762228945128,0) (6.283185307179586,0.5952532997570492)}; \addplot [color=black, thick] coordinates {(0,0.5952532997570492) (3.5886855397743136,6.283185307179586)}; \addplot [color=black, thick] coordinates {(3.5886855397743136,0) (6.283185307179586,4.270681089539234)}; \addplot [color=black, thick] coordinates {(0,4.270681089539234) (1.269748790097342,6.283185307179586)}; \addplot [color=black, thick] coordinates {(1.269748790097342,0) (5.233997347599963,6.283185307179586)}; \addplot [color=black, thick] coordinates {(5.233997347599963,0) (6.283185307179586,1.6629235721418438)}; \addplot [color=black, thick] coordinates {(0,1.6629235721418438) (2.915060597922997,6.283185307179586)}; \addplot [color=black, thick] coordinates {(2.915060597922997,0) (6.283185307179586,5.338351361924051)}; \addplot [color=black, thick] coordinates {(0,5.338351361924051) (0.5961238482460197,6.283185307179586)}; \addplot [color=black, thick] coordinates {(0.5961238482460197,0) (4.56037240574864,6.283185307179586)}; \addplot [color=black, thick] coordinates {(4.56037240574864,0) (6.283185307179586,2.730593844526672)}; \addplot [color=black, thick] coordinates {(0,2.730593844526672) (2.2414356560716744,6.283185307179586)}; \addplot [color=black, thick] coordinates {(2.2414356560716744,0) (6.205684213574295,6.283185307179586)}; \addplot [color=black, thick] coordinates {(6.205684213574295,0) (6.283185307179586,0.12283632712927034)};
			
			\addplot[thick,-,color=blue] coordinates {(0,0) (6.3,0)} node[pos=0.5,below] {$\phantom{0}\pi\phantom{0}$} node[pos=1,below,] {$2\pi$}; \addplot[thick,-,color=blue] coordinates {(0,6.3) (6.3,6.3)};
			
			\addplot[thick,-,color=red] coordinates {(0,0) (0,6.3)} node[pos=0.5,left] {$\pi$} node[pos=1,left,] {$2\pi$}; \addplot[thick,-,color=red] coordinates {(6.3,0) (6.3,6.3)};
		\end{axis}
	\end{tikzpicture}
	\caption{Plot of $\tau \mapsto (2^{i\tau},3^{i\tau})$ for $0\leq \tau \leq \frac{24\pi}{\log{2}}$. Here we identify the {\color{blue}top} and {\color{blue}bottom} edges and the {\color{red}left} and {\color{red}right} edges of the square.} 
\label{fig:L} \end{figure}

The Saksman--Seip construction (essentially contained in \cite{SS2009}*{Lemma 3}) can be formulated as follows.
\begin{theorem}[Saksman--Seip] \label{thm:SSconstruct} 
	For every $0<\delta<1$ and every open set $U$ on $\mathbb{T}^2$, there is a function $f$ in $\mathscr{H}^\infty$ of the form $f(s)=F(2^{-s},3^{-s})$ such that 
	\begin{enumerate}
		\item[(i)] $|f^\ast(\chi)|=1$ for almost every $\chi$ in $U$; 
		\item[(ii)] $|f^\ast(\chi)|=\delta$ for almost every $\chi$ in $\mathbb{T}^2 \setminus U$; 
		\item[(iii)] $|f(i\tau)|=1$ for almost every $\tau$ such that $\mathfrak{p}^{-i\tau}$ is in $\mathscr{L}\cap U$; 
		\item[(iv)] $|f(i\tau)|=\delta$ for almost every $\tau$ such that $\mathfrak{p}^{-i\tau}$ is in $\mathscr{L}\cap \left(\mathbb{T}^2 \setminus \overline{U}\right)$. 
	\end{enumerate}
\end{theorem}

We stress that in statements (iii) and (iv), we mean ``almost every'' with respect to the Lebesgue measure on $\mathbb{R}$. Note also that we here interpret $\mathfrak{p}^{-i\tau}$ as $(2^{-i\tau},3^{-i\tau})$.

The following elementary estimate will be used to control $\log|f_\xi|$. We omit the proof, which is similar to \cite{BP2021}*{Lemma~2.3}.
\begin{lemma}\label{lem:logxi} 
	If $z$ and $\xi$ are distinct point in $\mathbb{D}$, then
	\[-\frac{1}{2}\frac{(1-|\xi|^2)(1-|z|^2)}{|\xi-z|^2} \leq \log\left|\frac{\xi-z}{1-\overline{\xi}z}\right| \leq -\frac{1}{2}\frac{(1-|\xi|^2)(1-|z|^2)}{|1-\overline{\xi}z|^2}.\]
\end{lemma}

To prove Theorem~\ref{thm:largelim} we will use Theorem~\ref{thm:SSconstruct} in exactly the same way as Saksman and Seip used it in the proof of Theorem~\ref{thm:SS}~(a).
\begin{proof}
	[Proof of Theorem~\ref{thm:largelim}] Fix $0<\delta<1$ to be specified later and let $U_\delta$ be an open set on $\mathbb{T}^2$ that contains $\mathscr{L}$ and that satisfies $m_2(U_\delta) \leq \delta$. We will use the Dirichlet series constructed in Theorem~\ref{thm:SSconstruct} for $U=U_\delta$ multiplied by $5^{-s}$ to ensure that $f(+\infty)=0$. Note that the factor $5^{-s}$ has no effect on the modulus of the boundary values. Theorem~\ref{thm:SSconstruct}~(iii) guarantees that $f$ is a singular inner function (viewed as an element of $H^\infty(\mathbb{C}_0)$) and, in particular, that it maps $\mathbb{C}_0$ to $\mathbb{D}$.
	
	The first assertion \eqref{eq:largelim2} now follows at once from Theorem~\ref{thm:Nfswap}, since $f_\xi$ is inner for every $\xi$, so that $|f_\xi|=1$ almost everywhere on $i\mathbb{R}$.
	
	We continue with the second assertion \eqref{eq:largelim1}. From the definition of the mean counting function, Theorem~\ref{thm:jensen}, and Lemma~\ref{lem:jessenchi} (with $f(+\infty)=0$) we obtain
	\[\mathscr{M}_f(\xi) \leq \log{\frac{1}{|\xi|}} + \int_{\mathbb{T}^\infty} \log|f_\xi^\ast(\chi)|\,dm_\infty(\chi)\]
	for every $\xi \neq 0$. We use Theorem~\ref{thm:SSconstruct}~(i) to infer that
	\[\int_{\mathbb{T}^\infty} \log|f^\ast_\xi(\chi)|\,dm_\infty(\chi) = \int_{\mathbb{T}^2} \log|f^\ast_\xi(\chi)|\,dm_2(\chi) = \int_{\mathbb{T}^2 \setminus U_\delta} \log|f^\ast_\xi(\chi)|\,dm_2(\chi),\]
	since $|f^\ast_\xi(\chi)|=1$ if and only if $|f^\ast(\chi)|=1$. Using Theorem~\ref{thm:SSconstruct}~(ii), the upper bound in Lemma~\ref{lem:logxi}, and the estimate $|1-\overline{\xi} z|^2 \leq (1+|z|)^2$, we find that
	\[\log|f^\ast_\xi(\chi)| \leq -\frac{1}{2}\left(1-|f^\ast_\xi(\chi)|^2\right) \leq -\frac{1-\delta}{1+\delta} \frac{1-|\xi|^2}{2}\]
	for almost every $\chi$ in $\mathbb{T}^2 \setminus U_\delta$. Since $m_2(\mathbb{T} \setminus U_\delta) \geq 1 - \delta$, we obtain the stated upper bound \eqref{eq:largelim1} upon choosing $\delta$ such that
	\[1-\varepsilon = \frac{1-\delta}{1+\delta} (1-\delta). \qedhere\]
\end{proof}

In the proof of Theorem~\ref{thm:nolim} we will also need following version of the ergodic theorem \cite{QQ2020}*{Theorem~2.2.5}. If $U$ is an open set in $\mathbb{T}^\infty$ such that $m_\infty(\partial U) = 0$, then 
\begin{equation}\label{eq:ergodicopen} 
	\lim_{T \to \infty} \frac{1}{2T} \int_{-T}^T \mathbf{1}_U(\mathfrak{p}^{-i\tau})\,d\tau = m_\infty(U). 
\end{equation}

In the upcoming proof, we will use Theorem~\ref{thm:SSconstruct} similarly to how Saksman and Seip in the proof of Theorem~\ref{thm:SS}~(b).
\begin{figure}
	\centering 
	\begin{tikzpicture}
		\begin{axis}
			[ xmin = -1, ymin = -1, hide axis, axis line style=thin, xtick={0,3.1415,6.2832}, xticklabels={$0$,$\pi$,$2\pi$}, ytick={0,3.1415,6.2832}, yticklabels={$0$,$\pi$,$2\pi$}, ]
			
			\addplot [color=black, thick] coordinates {(0,0) (3.964248557502618,6.283185307179586)}; \addplot [color=black, thick] coordinates {(3.964248557502618,0) (6.283185307179586,3.6754277897821974)}; \addplot [color=black, thick] coordinates {(0,3.6754277897821974) (1.645311807825649,6.283185307179586)}; \addplot [color=black, thick] coordinates {(1.645311807825649,0) (5.609560365328268,6.283185307179586)}; \addplot [color=black, thick] coordinates {(5.609560365328268,0) (6.283185307179586,1.0676702723848086)}; \addplot [color=black, thick] coordinates {(0,1.0676702723848086) (3.290623615651298,6.283185307179586)}; \addplot [color=black, thick] coordinates {(3.290623615651298,0) (6.283185307179586,4.743098062167007)}; \addplot [color=black, thick] coordinates {(0,4.743098062167007) (0.9716868659743294,6.283185307179586)}; \addplot [color=black, thick] coordinates {(0.9716868659743294,0) (4.93593542347695,6.283185307179586)}; \addplot [color=black, thick] coordinates {(4.93593542347695,0) (6.283185307179586,2.135340544769617)}; \addplot [color=black, thick] coordinates {(0,2.135340544769617) (2.6169986737999813,6.283185307179586)}; \addplot [color=black, thick] coordinates {(2.6169986737999813,0) (6.283185307179586,5.810768334551819)}; \addplot [color=black, thick] coordinates {(0,5.810768334551819) (0.29806192412300986,6.283185307179586)}; \addplot [color=black, thick] coordinates {(0.29806192412300986,0) (4.26231048162563,6.283185307179586)}; \addplot [color=black, thick] coordinates {(4.26231048162563,0) (6.283185307179586,3.2030108171544285)}; \addplot [color=black, thick] coordinates {(0,3.2030108171544285) (1.9433737319486588,6.283185307179586)}; \addplot [color=black, thick] coordinates {(1.9433737319486588,0) (5.90762228945128,6.283185307179586)}; \addplot [color=black, thick] coordinates {(5.90762228945128,0) (6.283185307179586,0.5952532997570492)}; \addplot [color=black, thick] coordinates {(0,0.5952532997570492) (3.5886855397743136,6.283185307179586)}; \addplot [color=black, thick] coordinates {(3.5886855397743136,0) (6.283185307179586,4.270681089539234)}; \addplot [color=black, thick] coordinates {(0,4.270681089539234) (1.269748790097342,6.283185307179586)}; \addplot [color=black, thick] coordinates {(1.269748790097342,0) (5.233997347599963,6.283185307179586)}; \addplot [color=black, thick] coordinates {(5.233997347599963,0) (6.283185307179586,1.6629235721418438)}; \addplot [color=black, thick] coordinates {(0,1.6629235721418438) (2.915060597922997,6.283185307179586)}; \addplot [color=black, thick] coordinates {(2.915060597922997,0) (6.283185307179586,5.338351361924051)}; \addplot [color=black, thick] coordinates {(0,5.338351361924051) (0.5961238482460197,6.283185307179586)}; \addplot [color=black, thick] coordinates {(0.5961238482460197,0) (4.56037240574864,6.283185307179586)}; \addplot [color=black, thick] coordinates {(4.56037240574864,0) (6.283185307179586,2.730593844526672)}; \addplot [color=black, thick] coordinates {(0,2.730593844526672) (2.2414356560716744,6.283185307179586)}; \addplot [color=black, thick] coordinates {(2.2414356560716744,0) (6.205684213574295,6.283185307179586)}; \addplot [color=black, thick] coordinates {(6.205684213574295,0) (6.283185307179586,0.12283632712927034)};
			
			\addplot [color=red, thick] coordinates {(0,0) (3.964248557502618,6.283185307179586)}; \addplot [color=red, thick] coordinates {(3.964248557502618,0) (6.283185307179586,3.6754277897821974)}; \addplot [color=red, thick] coordinates {(0,3.6754277897821974) (1.645311807825649,6.283185307179586)}; \addplot [color=red, thick] coordinates {(1.645311807825649,0) (3.4208752206596476,2.814201426994364)}; \addplot [color=blue, thick, draw opacity=0, name path=t1] coordinates {(0,0) (3.258266985309866,6.283185307179586)}; \addplot [color=blue, thick, draw opacity=0, name path=t2] coordinates {(3.258266985309866,0) (3.6259906059690334,0.709111825656249)}; \addplot [color=blue, thick, draw opacity=0, name path=t3] coordinates {(3.6259906059690334,0.709111825656249) (4.796467479511955,2.966239043099107)}; \addplot [color=blue, thick, draw opacity=0, name path=t4] coordinates {(4.796467479511955,2.966239043099107) (5.457052434733744,4.240099565734793)}; \addplot [color=blue, thick, draw opacity=0, name path=b1] coordinates {(0,0) (4.2682026588138875,5.645986841572669)}; \addplot [color=blue, thick, draw opacity=0, name path=b2] coordinates {(4.2682026588138875,5.645986841572669) (4.749906258452106,6.283185307179586)}; \addplot [color=blue, thick, draw opacity=0, name path=b3] coordinates {(4.749906258452106,0) (6.283185307179586,2.028224530458475)}; \addplot [color=blue, thick, draw opacity=0, name path=b4] coordinates {(0,2.028224530458475) (0.8653405245678334,3.172898672184278)}; \addplot [color=blue, thick, draw opacity=0, name path=t5] coordinates {(5.457052434733744,4.240099565734793) (6.283185307179586,5.184944622015994)}; \addplot [color=blue, thick, draw opacity=0, name path=t6] coordinates {(0,5.184944622015994) (0.960255574011347,6.283185307179586)}; \addplot [color=blue, thick, draw opacity=0, name path=t7] coordinates {(0.960255574011347,0) (1.403496484975241,0.5069329613115787)}; \addplot [color=blue, thick, draw opacity=0, name path=t8] coordinates {(1.403496484975241,0.5069329613115787) (3.4208752206596476,2.814201426994366)}; \addplot [color=blue, thick, draw opacity=0, name path=b5] coordinates {(0.8653405245678334,3.172898672184278) (1.3624460569412333,4.325336867055657)}; \addplot [color=blue, thick, draw opacity=0, name path=b6] coordinates {(1.3624460569412333,4.325336867055657) (1.9402566903811083,5.664873444328659)}; \addplot [color=blue, thick, draw opacity=0, name path=b7] coordinates {(1.9402566903811083,5.664873444328659) (2.206966219033622,6.283185307179586)}; \addplot [color=blue, thick, draw opacity=0, name path=b8] coordinates {(2.206966219033622,0) (3.420875220659648,2.814201426994364)};
			
			\addplot[thick,-,color=blue,name path=bottom] coordinates {(0,0) (6.3,0)} node[pos=0.5,below] {$\phantom{0}\pi\phantom{0}$} node[pos=1,below,] {$2\pi$}; \addplot[thick,-,color=blue,name path=top] coordinates {(0,6.3) (6.3,6.3)};
			
			\addplot[thick,-,color=red,name path=left] coordinates {(0,0) (0,6.3)} node[pos=0.5,left] {$\pi$} node[pos=1,left,] {$2\pi$}; \addplot[thick,-,color=red,name path=right] coordinates {(6.3,0) (6.3,6.3)};
			
			\addplot[blue!50] fill between[of=t1 and b1]; \addplot[blue!50] fill between[of=t1 and b2]; \addplot[blue!50] fill between[of=t2 and b3]; \addplot[blue!50] fill between[of=t3 and b3]; \addplot[blue!50] fill between[of=t4 and b3]; \addplot[blue!50] fill between[of=t5 and b3]; \addplot[blue!50] fill between[of=t6 and b4]; \addplot[blue!50] fill between[of=t6 and b5]; \addplot[blue!50] fill between[of=t6 and b6]; \addplot[blue!50] fill between[of=t6 and b7]; \addplot[blue!50] fill between[of=t7 and b8];
		\end{axis}
	\end{tikzpicture}
	\caption{Plot of $\tau \mapsto (2^{i\tau},3^{i\tau})$ for $0\leq \tau \leq \frac{24\pi}{\log{2}}$, the {\color{red}segment} $\{(2^{i\tau},3^{i\tau})\,:\, 0 < \tau < 14\}$, and an open {\color{blue}{parallelogram}} containing the segment. Here we identify the {\color{blue}top} and {\color{blue}bottom} edges and the {\color{red}left} and {\color{red}right} edges of the square.} 
\label{fig:para} \end{figure}
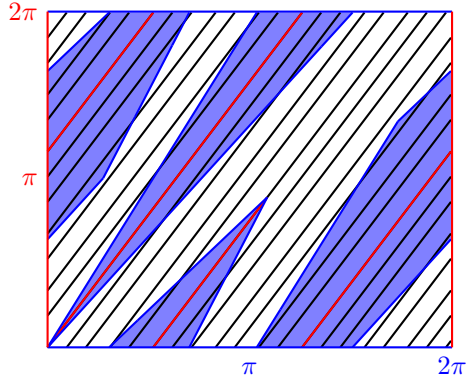
\begin{proof}
	[Proof of Theorem~\ref{thm:nolim}] Set
	\[\mathscr{L}_n = \left\{(2^{-i\tau},3^{-i\tau})\,:\, -n < \tau < n\right\}\]
	for $n=1,2,3,\ldots$ and let $U_n$ be an open set in $\mathbb{T}^2$ that satisfies 
	\begin{enumerate}
		\item[(i)] $\mathscr{L}_n$ is contained in $U_n$; 
		\item[(ii)] $m_2(U_n) \leq \delta 2^{-n-1}$; 
		\item[(iii)] $\partial U_n$ has countable intersection with $\mathscr{L}$ and $m_2(\partial U_n) = 0$. 
	\end{enumerate}
	To construct such a set, we could for example consider a parallelogram in $\mathbb{R}^2$ with one diagonal coinciding with the segment of the line $(\log 2)y = (\log 3)x$ corresponding to $\mathscr{L}_n$, and with the other diagonal as a suitably small perpendicular segment. The set $U_n$ is then obtained as the projection onto $\mathbb{R}/2\pi \times \mathbb{R}/2\pi$. See Figure~\ref{fig:para}. Properties (i) and (ii) then hold by construction, and (iii) holds since $\partial U_n$ consists of a finite number of line segments, none of which are parallel with $\mathscr{L}$. Note that while we have stipulated that $\mathscr{L}_n$ is contained in $U_n$, the open set $U_n$ will also intersect $\mathscr{L}$ in many other segments, since $\mathscr{L}$ is dense in $\mathbb{T}^2$.
	
	We next define $V_n = U_1 \cup U_2 \cup \cdots \cup U_n$. Clearly $m_2(V_n) \leq \delta/2$ and $m_2(\partial V_n) = 0$. For an increasing sequence $(n_k)_{k\geq1}$ of integers we set $W_k = V_{n_{k+1}} \setminus \overline{V_{n_k}}$. Then the properties of $U_n$ ensure that 
	\begin{enumerate}
		\item[(i)] $(W_k)_{k\geq1}$ is a disjoint sequence of open sets; 
		\item[(ii)] $\mathscr{L} \setminus \bigcup_{k\geq1} W_k$ is countable; 
		\item[(iii)] all but a countable subset of $\mathscr{L}_{n_{k+1}}$ is contained in either $W_k$ or in $V_{n_k}$. 
	\end{enumerate}
	We choose $(n_k)_{k\geq1}$ as follows. We let $n_1 =1$ and given $n_k$, we pick $n_{k+1} \geq 2 n_k$ such that 
	\begin{equation}\label{eq:ergodicuse} 
		\frac{1}{2 n_{k+1}} \int_{-n_{k+1}}^{n_{k+1}} \mathbf{1}_{V_{n_k}}(\mathfrak{p}^{-i\tau})\,d\tau \leq 2 m_2(V_{n_k}) \leq \delta, 
	\end{equation}
	which is always possible by \eqref{eq:ergodicopen}. The assumption $n_{k+1} \geq 2n_k$ is only included to ensure that $n_k \to \infty$ as $k\to\infty$.
	
	Let
	\[U = \bigcup_{k=1}^\infty W_{2k-1} \qquad \text{and} \qquad \widetilde{U} = \bigcup_{k=1}^\infty W_{2k} \subseteq \mathbb{T}^2 \setminus \overline{U}.\]
	It follows from this construction and Theorem~\ref{thm:SSconstruct} that there is a function $f$ in $\mathscr{H}^\infty$ with $f(+\infty)=0$ such that $|f(i\tau)|=1$ for almost every $\tau$ such that $\mathfrak{p}^{-i\tau}$ is in $U$ and $|f(i\tau)|=\varepsilon/2$ for almost every $\tau$ such that $\mathfrak{p}^{-i\tau}$ is in $\widetilde{U}$, and, furthermore, almost every $\tau$ falls within one of these cases.
	
	In preparation for the needed estimates, note that if $0<\varepsilon<|\xi|<1$ and $|f(i\tau)|=\varepsilon/2$, then from Lemma~\ref{lem:logxi} we have that 
	\begin{equation}\label{eq:xieps} 
		-c \frac{1-|\xi|^2}{2} \leq \log|f_\xi(i\tau)| \leq - C \frac{1-|\xi|^2}{2}, 
	\end{equation}
	for $c=(2/\varepsilon)^2-1$ and $C=(2-\varepsilon)/(2+\varepsilon)$. The lower bound also holds when $|f(i\tau)|=1$. 
	
	Since $\mathbf{1}_{W_{2k}}(\mathfrak{p}^{-i\tau}) + \mathbf{1}_{V_{n_{2k}}}(\mathfrak{p}^{-i\tau}) = 1$ for almost every $\tau$ in $(-n_{2k+1}, n_{2k+1})$, we have by \eqref{eq:ergodicuse} that
	\[ \frac{1}{2 n_{2k+1}} \int_{-n_{2k+1}}^{n_{2k+1}} \mathbf{1}_{W_{2k}}(\mathfrak{p}^{-i\tau})\,d\tau \geq 1- \delta, \]
	By the upper bound in \eqref{eq:xieps} we thus infer that 
	\begin{multline*}
		\frac{1}{2 n_{2k+1}} \int_{-n_{2k+1}}^{n_{2k+1}} \log|f_\xi(i\tau)| d\tau \\
		\leq \frac{1}{2 n_{2k+1}} \int_{-n_{2k+1}}^{n_{2k+1}} \mathbf{1}_{W_{2k}}(\mathfrak{p}^{-i\tau}) \log|f_\xi(i\tau)| \,d\tau \leq -(1-\delta) C\frac{1-|\xi|^2}{2}. 
	\end{multline*}
	
	Similarly, since $\mathbf{1}_{W_{2k-1}}(\mathfrak{p}^{-i\tau}) + \mathbf{1}_{V_{n_{2k-1}}}(\mathfrak{p}^{-i\tau}) = 1$ for almost every $\tau$ in $(-n_{2k}, n_{2k})$, we have that 
	\begin{multline*}
		\frac{1}{2 n_{2k}} \int_{-n_{2k}}^{n_{2k}} \log|f_\xi(i\tau)| d\tau \\
		= \frac{1}{2 n_{2k}} \int_{-n_{2k}}^{n_{2k}} \mathbf{1}_{V_{n_{2k-1}}}(\mathfrak{p}^{-i\tau}) \log|f_\xi(i\tau)| \, d\tau \geq - \delta c \frac{1-|\xi|^2}{2}, 
	\end{multline*}
	where we in the final estimate used \eqref{eq:ergodicuse}, and the lower bound in \eqref{eq:xieps}, which holds for almost every $\tau$.
	
	If $\delta$ is so small that $(1-\delta) C> c\delta$, then it follows from what we have done and Theorem~\ref{thm:Nfswap} that
	\[\limsup_{T\to\infty} N_f(\xi,T) \geq - \delta c \frac{1-|\xi|^2}{2} > -(1-\delta) C\frac{1-|\xi|^2}{2} \geq \liminf_{T\to\infty} N_f(\xi,T)\]
	for quasi-every $\xi$ with $\varepsilon<|\xi|<1$. 
\end{proof}

In comparison with Theorem~\ref{thm:Hpchiswap} and Theorem~\ref{thm:LPchiswap}, let us mention the following result (see \cite{KP2023}*{Theorem~4.9} for the proof of the second assertion).
\begin{theorem}\label{thm:mcfchiswap} 
	If $f$ is in $\mathscr{H}^\infty$, then $\mathscr{M}_f(\xi) = \mathscr{M}_{f_\chi}(\xi)$ for every $\chi$ on $\mathbb{T}^\infty$. Moreover,
	\[\mathscr{M}_f(\xi) = \lim_{T\to\infty}\frac{\pi}{T} \sum_{\substack{s \in f_\chi^{-1}(\{\xi\}) \\
	|\mim{s}| < T \\
	0<\mre{s}<\infty}}\mre{s}\]
	for almost every $\chi$ on $\mathbb{T}^\infty$. 
\end{theorem}

\bibliography{apbv}

\end{document}